\documentclass[a4paper]{article}

\usepackage{amsmath,amsfonts,amssymb,amscd,amsthm,amsbsy}
\usepackage[latin1]{inputenc}   % pour les accents

\usepackage[a4paper]{geometry}
\usepackage{tikz}

\usepackage{pgf,tikz}
\usetikzlibrary{arrows}

\def\ds{\displaystyle}
\def\rg{\rangle}
\def\lg{\langle} 
\def \to{\rightarrow}
\def \states{\mathbb{T}^d}
\def \T{\mathbb{T}}
\def \R {\mathbb{R}}
\def \N {\mathbb{N}}
\def \Z {\mathbb{Z}}
\def \mes{\mathcal{P}}
\def\Pk{\mes(\states)}
\def \mint{\int_{\states}}

\def \d{\mathrm{d}}
\def\dk{{\bf d}_1}
\def \abs{|}

\def \leftnorm{\left|}
\def \rightnorm{\right|}
\def \labs{\big|}
\def \ep{\varepsilon}
\newcommand{\be}{\begin{equation}}
\newcommand{\ee}{\end{equation}}

\newtheorem{Theorem}{Theorem}[section]
\newtheorem{Definition}[Theorem]{Definition}
\newtheorem{Proposition}[Theorem]{Proposition}
\newtheorem{Lemma}[Theorem]{Lemma}
\newtheorem{Corollary}[Theorem]{Corollary}
\newtheorem{Remark}[Theorem]{Remark}

\title{Learning in Mean Field Games: the Fictitious Play}
\author{Pierre Cardaliaguet\thanks{Ceremade, Université Paris-Dauphine} \and Saeed Hadikhanloo\thanks{Lamsade, Université Paris-Dauphine}}

\begin{document}

\maketitle

\begin{abstract}
Mean Field Game systems describe equilibrium configurations in differential games with infinitely many infinitesimal interacting agents. We introduce a learning procedure (similar to the Fictitious Play) for these games and show its convergence when the Mean Field Game is  potential. 
\end{abstract}

\tableofcontents

\section{Introduction}

Mean Field Game is a class of differential games in which each agent is infinitesimal and interacts with a huge population of other agents. These games have been introduced simultaneously by Lasry, Lions \cite{LL06cr1, LL06cr2, LL07mf, LLperso} and Huang, Malhamé and Caines \cite{HCMieeeAC06}, (actually a discrete in time version of these games were previously known under the terminology of heterogenous models in economics. See for instance  \cite{Aiy}). The classical notion of solution in Mean Field Game (abbreviated MFG) is given by a pair of maps $(u,m)$, where $u=u(t,x)$ is the value function of a typical small player while $m=m(t,x)$ denotes the density at time $t$ and at position $x$ of the population. The value function $u$ satisfies a Hamilton-Jacobi equation---in which $m$ enters as a parameter and describes the influence of the population on the cost of each agent---, while the density $m$ evolves in time according to a Fokker-Planck equation in which $u$ enters as a drift. More precisely the pair $(u,m)$ is a solution of the {\it MFG system}, which reads 
\be\label{MFGintrtro}
\left\{\begin{split}
(i) \qquad &-\partial_t u - \sigma \Delta u  +H(x,\nabla u(t,x)) =f(x,m(t)) \\
(ii) \qquad &\partial_t m- \sigma \Delta m  - \text{div} ( mD_p H (x,\nabla u)) =0  \\
&m(0,x)=m_0(x) , \; u(T,x)=g(x, m (T)).
\end{split}\right.
\ee
In the above system, $T>0$ is the horizon of the game, $\sigma$ is a nonnegative parameter describing the intensity of the (individual) noise each agent is submitted to (for simplicity we assume that either $\sigma=0$ (no noise) or $\sigma=1$, some individual noise). The map $H$ is the Hamiltonian of the control problem (thus typically convex in the gradient variable). The running cost $f$ and the terminal cost $g$ depend on the one hand on the position of the agent and, on the other hand, on the population density. Note that, in order to solve the (backward) Hamilton-Jacobi equation (i.e., the optimal control problem of each agent) one has to know the evolution of the population density, while the Fokker-Planck equation depends on the optimal strategies of the agents (through the drift term $- \text{div} ( mD_p H (x,\nabla u))$). The MFG system formalizes therefore  {\it an equilibrium configuration}. 

Under suitable assumptions recalled below, the MFG system~\eqref{MFGintrtro} has at least one solution. This solution is even unique under a monotonicity condition on $f$ and $g$. Under this condition, one can also show that it is the limit of symmetric Nash equilibria for a finite number of players  as the number of players tends to infinity \cite{CDLL}; moreover, the optimal strategy given by the solution of the MFG system can be implemented in the game with finitely many players to give an approximate Nash equilibrium \cite{HCMieeeAC06, CaDe2}. MFG systems have been widely used  in several areas ranging from engineering to economics, either under the terminology of heterogeneous agent model \cite{Aiy, Bew, Hug}, or under the name of MFG  \cite{ABLLM, ALLM, GLL2010,HCMdec2003}.\\

In the present paper we raise the question of the actual formation of the MFG equilibrium. Indeed, the game being quite involve, it is unrealistic to assume that the agents can actually {\it compute} the equilibrium configuration. This seems to indicate that, if the equilibrium configuration arises, it is because the agents have {\it learned} how to play the game.  For instance, people driving every day from home to work are dealing with such a learning issue. Every day they try to forecast the traffic and choose their optimal path accordingly, minimizing the journey and/or the consumed fuel for instance. If their belief on the traffic turns out not to be correct, they update their estimation, and so on... The question is wether such a procedure leads to stability or not. 

The question of {\it learning}  is a very classical one in game theory (see, for instance, the monograph \cite{FL98}). There is by now a very large number of learning procedures for one-shot games in the literature. In the present paper we focus on a very classical and simple one: the Fictitious Play. The Fictitious Play was first introduced by Brown \cite{BrownFictitiousPlay}. In this learning procedure,  every player plays at each step the best response action with respect to the average of the previous actions of the other players. Fictitious Play does not necessarily converge, as shows the counter-example by Shapley \cite{Sha64}, but it is known to converge for several classes of one shot games: for instance for zero-sum games (Robinson \cite{RO51}), for $2\times 2$ games (Miyasawa \cite{Mi61}), for potential games (Monderer and Shapley \cite{MondrerShapley2})... \\
%In particular, Monderer and Shapley \cite{MondrerShapley2} showed that this procedure converges to the Nash equilibrium in the case of Potential $N-$player finite games.\\

Note that, in our setting, the question of learning makes all the more sense that the game is particularly intricate. Our aim is to define a Fictitious Play for the MFG system and to prove the convergence of this procedure under suitable assumption on the coupling $f$ and $g$. The Fictitious Play for the MFG system runs as follows: the players start with a smooth initial belief $(m^0(t))_{t\in [0,T]}$. At the beginning of stage $n+1$, the players having observed the same past, share the same belief $(\overline{m}^{n}(t))_{t\in [0,T]}$ on the evolving density of the population. They compute their corresponding optimal control problem with value function $u^{n+1}$ accordingly.
When all players actually implement their optimal control the population density evolves in time and the players observe the resulting evolution $(m^{n+1}(t))_{t\in [0,T]}$. At the end of stage $n+1$ the players update their belief according to the rule (the same for all the players), which consists in computing the {\it average} of their observation up to time $n+1$. This yields to define by induction the sequences $u^n,m^n,\bar m^n$ by:
\begin{equation}\label{MFG2FPintro}
\left\{\begin{split}
(i) \qquad &-\partial_t u^{n+1} - \sigma\Delta u^{n+1}  +H(x,\nabla u^{n+1}(t,x)) =f(x,\bar{m}^{n}(t)), \\
(ii) \qquad &\partial_t m^{n+1} - \sigma\Delta m^{n+1}  - \text{div} (  m^{n+1}D_p H (x,\nabla u^{n+1})) =0 ,  \\
&m^{n+1}(0)=m_0 , \; u^{n+1}(x,T)=g(x, \bar{m} ^n (T))
\end{split} \right. 
\end{equation}
where $\bar{m}^{n} = \frac{1}{n} \sum_{k=1}^{n} {m}^{k}$. Indeed, $u^{n+1}$ is the value function at stage $n+1$ if the belief of players on the evolving density is $\bar m^n$, and thus solves \eqref{MFG2FPintro}-(i). The actual density then evolves according to the Fokker-Planck equation \eqref{MFG2FPintro}-(ii). 

Our main result is that, under suitable assumption, this learning procedure converges, i.e., any cluster point of the pre-compact sequence $(u^n, m^n)$ is a solution of the MFG system~\eqref{MFGintrtro} (by compact, we mean compact for the uniform convergence).  Of course, if in addition the solution of the MFG system~\eqref{MFGintrtro} is unique, then the full sequence converges. Let us recall (see \cite{LL07mf}) that this uniqueness holds for instance if $f$ and $g$ are monotone:  
$$
\int (f(x,m)-f(x,m')\ \d (m-m')(x)\geq 0, \qquad \int (g(x,m)-g(x,m')\ \d (m-m')(x)\geq 0
$$
for any probability measure $m,m'$. This condition is generally interpreted as an aversion for congestion for the agents. Our key assumptions for the convergence result is that  {\it  $f$ and $g$ derive from potentials}.
By this we mean that there exists $F=F(m)$ and $G=G(m)$ such that 
$$
f(x,m)= \frac{\delta F}{\delta m}(x,m)\qquad {\rm and}\qquad g(x,m)= \frac{\delta G}{\delta m}(x,m).
$$
The above derivative---in the space of measure---is introduced in subsection \ref{sec.PotMFG}, the definition being borrowed from \cite{CDLL}. Our assumption actually ensures  that our MFG system is also ``a potential game" (in the flavor of Monderer and Shapley \cite{MondrerShapley1}) so that the MFG system falls into a framework closely related to that of Monderer and Shapley \cite{MondrerShapley2}.  Compared to \cite{MondrerShapley2}, however, we face two issues. First we have an infinite population of players and the state space and the actions are also infinite. Second the game has a much more involve structure than in \cite{MondrerShapley2}. In particular, the potential for our game is far from being straightforward. We consider two different frameworks. In the first one, the so-called second order MFG systems where $\sigma =1$---which corresponds to the case where the players have a dynamic perturbed by independent noise---the potential is defined as a map of the evolving population density. This is reminiscent of the variational structure for the MFG system as introduced in \cite{LL07mf} and exploited in \cite{c1, CGPT} for instance. The proof of the convergence then strongly relies on the regularity properties of the value function and of the population density (i.e., of the $u^n$ and $m^n$). The second framework is for first order MFG systems, where $\sigma=0$. In contrast with the previous case, the lack of regularity of the value function and of the population density prevent to define the same Fictitious Play and the same potential. To overcome the difficulty, we lift the problem to the space of curves, which is the natural space of strategies. We define the Fictitious Play and a potential in this setting, and then prove the convergence, first for the infinite population and then for a large, but finite, one. 

As far as we are aware of, our paper is the first one to consider a learning procedure in the framework of mean field games. Let us nevertheless point out that, for a particular class of MFG systems (quadratic Hamiltonians, local coupling), Guéant introduces in \cite{Gu12} an algorithm which is closely related to a replicator dynamics: namely it is exactly \eqref{MFG2FPintro} in which one replaces $\bar m^n$ by $m^n$ in \eqref{MFG2FPintro}-(i)). The convergence is proved by using a kind of monotonicity of the sequence. This monotonicity does not hold in the more intricate framework considered here.  \\

For simplicity we work in the periodic setting: we assume that the maps $H$, $f$ and $g$ are periodic in the space variable (and thus actually defined on the torus $\T^d=\R^d/\Z^d$). This simplifies the estimates and the notation. However we do not think that the result changes in a substantial way if the state space is $\R^d$ or a subdomain of $\R^d$, with suitable boundary conditions. \\

The paper is organized as follows: we complete the introduction by fixing the main notation and stating the basic assumptions on the data. Then we define the notion of potential MFG and characterize the conditions of deriving from a potential. Section \ref{sec.MFG2} is devoted to the Fictitious Play for second order MFG systems while section \ref{sec.MFG1} deals with the first order ones. \\

{\bf Acknowledgement:} The first author was partially supported by the ANR (Agence Nationale de la Recherche) projects  ANR-10-BLAN 0112, ANR-12-BS01-0008-01 and ANR-14-ACHN-0030-01.

%%%%%%%%%%%%%%%%%%%%%%%
\subsection{Preliminaries and Assumptions}

If $X$ is a metric space, we denote by $\mes(X)$ the set of Borel probability measures on $X$. When $X=\T^d$ ($\T^d$ being the torus $\R^d/\Z^d$), we endow $\mes(\T^d)$ with the distance
\begin{equation}\label{MWdistance}
\dk (\mu,\nu)= \sup_h \left\{ \int_{\T^d} h(x)\ \d (\mu-\nu)(x) \right\}\qquad \mu,\nu \in \mes (\T^d), 
\end{equation}
where the supremum is taken over all the maps $h:\T^d\to \R$ which are 1-Lipschitz continuous. Then $\dk$ metricizes the weak-* convergence of measures on $\T^d$. 

The maps $H$, $f$ and $g$ are periodic in the space arguments: $H:\T^d\times \R^d\to \R$ while $f,g:\T^d\times \mes(\states)\to\R$. In the same way, the initial condition $m_0\in \Pk$ is periodic in space and  is assumed to be absolutely continuous with a smooth density. 
 
We now state our key assumptions on the data: these conditions are valid throughout the paper. On the initial measure $m_0$, we assume that
\begin{equation}\label{hypm0}
\mbox{\rm  $m_0$ has a smooth density (again denoted $m_0$).}
\end{equation}
Concerning the Hamiltonian, 
we suppose that $H$  is of class ${\mathcal C}^2$ on $\T^d\times \R^d$ and quadratic-like in the second variable: 
\begin{equation}\label{hyp:unifCv}
H\in {\mathcal C}^2(\T^d\times \R^d)\; {\rm and }\; \frac{1}{\bar C} I_d\leq D^2_{pp} H(x,p) \leq \bar C I_d \qquad \forall (x,p)\in \T^d\times \R^d\;.
\end{equation}
Moreover, we suppose that $D_xH$ satisfies the lower bound: 
\be\label{hyp:Hcondsupp}
\lg D_xH(x,p), p\rg \geq -C \left(|p|^2+1\right).
\ee
The maps $f$ and $g$ are  supposed to be globally Lipschitz continuous (in both variables) and regularizing:
\begin{equation}\label{regucondf}
\begin{array}{c}
\ds \mbox{\rm The map $m\to f(\cdot,m)$ is Lipschitz continuous from $\mes (\T^d)$ to ${\mathcal C}^2(\T^d)$}\\
\ds \mbox{\rm while the map $m\to g(\cdot,m)$ is Lipschitz continuous from $\mes (\T^d)$ to ${\mathcal C}^3(\T^d)$.}
\end{array}\end{equation}
 In particular, there is $\bar C>0$ such that
\begin{equation}\label{boundC2}
\sup_{m\in P(\T^d)} \left\| f(\cdot,m)\right\|_{{\mathcal C}^2} +\left\| g(\cdot,m)\right\|_{{\mathcal C}^3} \leq \bar C.
\end{equation}
Assumptions \eqref{hypm0}, \eqref{hyp:unifCv}, \eqref{hyp:Hcondsupp}, \eqref{regucondf}, \eqref{mono} are in force throughout the paper. As explained below, they ensure the MFG system to have at least one solution. 

To ensure the uniqueness of the solution, we sometime require $f$ and $g$ to be monotone: for any $m,m'\in \mes(\states)$, 
\begin{equation}\label{mono}
\int_{\states} (f(x,m)-f(x,m'))d(m-m')(x)\geq 0, \; \int_{\states} (g(x,m)-g(x,m'))d(m-m')(x)\geq 0.
\end{equation}
This condition can be interpreted as a dislike of congested area by the agent.

%%%%%%%%%%%%%%%%%%%%%%%%%%%%%%%%%%%
%%%%%%%%%%%%%%%%%%%%%%%%%%%%%%%%%%
\subsection{ Potential Mean Field Games}\label{sec.PotMFG}

In this section we introduce the main structure condition on the data $f$ and $g$ of the game: we assume that $f$ and $g$ are the derivative, with respect to the measure, of potential maps $F$ and $G$. In this case we say that $f$ and $g$ derive from a potential. 

Let us first explain what we mean by a derivative with respect to a measure. Let $F:\mes(\states)\to \R$ be a continuous map. We say that the continuous map $\frac{\delta F}{\delta m}:\states\times \mes(\states)\to \R$ is the derivative of $F$ if, for any $m,m'\in \mes(\states)$, 
\begin{equation}\label{Potential}
\lim_{s\to 0} \frac{F((1-s)m+sm')-F(m)}{s} = \int_{\states} \frac{\delta F}{\delta m}(m,x)\d (m'-m)(x).
\end{equation}
As $\frac{\delta F}{\delta m}$ is continuous, this equality can be equivalently written as 
$$
F(m')-F(m)= \int_0^1\mint \frac{\delta F}{\delta m}((1-s)m+sm'),x)\d (m'-m)(x) \d s, 
$$  
for any $m,m'\in \mes(\states)$. 
Note that $\frac{\delta F}{\delta m}$ is defined only up to an additive constant. To fix the ideas we assume therefore that 
$$
\int_{\T^d} \frac{\delta F}{\delta m}(m,x)\d m(x)=0\qquad \forall m\in \mes(\T^d).
$$
We often use the notation
$$
\frac{\delta F}{\delta m}(m)(m'-m):=  \int_{\states} \frac{\delta F}{\delta m}(x,m)d(m'-m)(x).
$$

\begin{Definition}
A Mean Field Game is called a Potential Mean Field Game if the instantaneous and final cost functions $f,g : \states \times \mes(\states) \to \R$ derive from potentials, i.e., there exists $F,G : \mes(\states) \to \R$ such that
$$
\frac{\delta F}{\delta m}= f, \qquad  \frac{\delta G}{\delta m}= g.
$$
\end{Definition}

In the rest of the section we characterize the maps $f$ which derive from a potential. Although this is not used in the rest of the paper, this characterization is natural and we believe that it has its own interest. \\

To proceed we assume for the rest of the section that, for any $x\in \states$, $f(x,\cdot)$ has a derivative and that this derivative $\frac{\delta f}{\delta m}:\states \times \mes(\states)\times \states\to \R$ is continuous. 
%Then, for any $m,m'\in \mes(\states)$, 
%$$
%f(x, (1-s)m + sm' ) = f(x,m) + s \int_{\states} \frac{\delta f}{\delta m } (x,m,y) d(m'-m)(y) + o(s),
%$$
%where $\lim_{s \to 0} \frac{o(s)}{s}=0$. 

\begin{Proposition}\label{prop.potential}
The map $f:\states\times \mes(\states)\to \R$ derives from a potential, if and only if, 
$$
\ds \frac{\delta f}{\delta m}(y,m,x)= \frac{\delta f}{\delta m}(x,m,y)\qquad \forall x,y\in \T^d, \; \forall m \in \mes(\T^d).
$$ 
\end{Proposition}

\begin{proof} First assume that $f$ derives from a potential $F:\mes(\states)\to \R$. Deriving in $m$ the relation $\frac{\delta F}{\delta m}=f$ we obtain
$$
\frac{\delta^2 F}{\delta m^2}(m,x,y)= \frac{\delta f}{\delta m}(x,m,y)\qquad \forall x,y\in \states, \; m\in \mes(\states).
$$
As $\frac{\delta^2 F}{\delta m^2}(m,x,y)$ is symmetric in $(x,y)$ (see \cite{CDLL}), so is $\frac{\delta f}{\delta m}(x,m,y)$. 

Let us now assume that $\frac{\delta f}{\delta m}(x,m,y)$ is symmetric in $(x,y)$. Let us fix $m_0\in \mes(\states)$ and set, for any $m\in \mes(\states)$, 
$$
F(m)=\int_0^1 \int_{\states} f(x,(1-t)m_0+tm) {\rm d}(m-m_0)(x) {\rm d}t.
$$
We claim that $F$ is a potential for $f$. Indeed, as $f$ has a continuous derivative, so has $F$, with 
\be\label{hazberosuldi}
\begin{array}{rl}
\ds \frac{\delta F}{\delta m}(m,y)\; = & \ds \int_0^1 t \int_{\states} \frac{\delta f}{\delta m}(x,(1-t)m_0+tm,y)\ {\rm d}(m-m_0)(x) {\rm d}t\\
& \qquad \ds + 
\int_0^1 f(y,(1-t)m_0+tm) {\rm d}t.
\end{array}
\ee
As, by symmetry assumption,
$$
\begin{array}{rl}
\ds \frac{d}{dt}  f(y,(1-t)m_0+tm)  \; = & \ds \int_{\states}\frac{\delta f}{\delta m}(y,(1-t)m_0+tm,x) d(m-m_0)(x)\\
= & \ds \int_{\states}\frac{\delta f}{\delta m}(x,(1-t)m_0+tm,y) d(m-m_0)(x) ,
\end{array}
$$
we have therefore after integration by parts in \eqref{hazberosuldi},  
$$
\frac{\delta F}{\delta m}(m,y)=\Big[ t\ f(x,(1-t)m_0+tm)\Big]_0^1 = f(x,m).
$$
\end{proof}

%%%%%%%%%%%%%%%%%%%%%%%%%%%%%%%%%%%%%%%
%%%%%%%%%%%%%%%%%%%%%%%%%%%%%%%%%%%%%%%
%%%%%%%%%%%%%%%%%%%%%%%%%%%%%%%%%%%%%%
%%%%%%%%%%%%%%%%%%%%%%%%%%%%%%%%%%%%%

\section{The Fictitious Play for second order MFG systems}\label{sec.MFG2}

In this section, we study a learning procedure for the second order MFG system: 
\begin{equation}\label{MFG2}
\left\{\begin{split}
(i) \qquad &-\partial_t u - \Delta u  +H(x,\nabla u(t,x)) =f(x,m(t)), \quad (t,x)\in  [0,T]\times \states  \\
(ii) \qquad &\partial_t m- \Delta m  - \text{div} ( mD_p H (x,\nabla u)) =0 , \quad (t,x)\in  [0,T]\times \states  \\
&m(0)=m_0 , \; u(x,T)=g(x, m (T)), \quad  x\in \states.
\end{split}\right.
\end{equation}
Let us recall (see \cite{LL07mf}) that, under our assumptions \eqref{hypm0}, \eqref{hyp:unifCv}, \eqref{hyp:Hcondsupp}, \eqref{regucondf}, there exists at least one classical solution to \eqref{MFG2} (i.e., for which all the involved derivative exists and are continuous). If furthermore \eqref{mono} holds, then the solution is unique. 

%%%%%%%%%%%%%%%%%%%%%%%%%%%
\subsection{The learning rule and the convergence result}

The Fictitious Play can be written as follows: given a smooth initial guess $m^0\in  C^0([0,T], \mes(\T^d))$, we define by induction sequences $u^n,m^n:[0,T]\times \T^d\to \R$ by:
\begin{equation}\label{MFG2FP}
\left\{\begin{split}
(i) \qquad &-\partial_t u^{n+1} - \Delta u^{n+1}  +H(x,\nabla u^{n+1}(t,x)) =f(x,\bar{m}^{n}(t)), \quad (t,x)\in  [0,T]\times \states  \\
(ii) \qquad &\partial_t m^{n+1} - \Delta m^{n+1}  - \text{div} (  m^{n+1}D_p H (x,\nabla u^{n+1})) =0 , \quad (t,x)\in  [0,T]\times \states  \\
&m^{n+1}(0)=m_0 , \; u^{n+1}(x,T)=g(x, \bar{m} ^n (T)), \quad  x\in \states
\end{split} \right. 
\end{equation}
where $\bar{m}^{n} (t,x) = \frac{1}{n} \sum_{k=1}^{n} {m}^{k}(t,x)$. The interpretation is that, at the beginning of stage $n+1$, the players have the same belief of the future density of the population  $(\overline{m}^{n}(t))_{t\in [0,T]}$ and compute their corresponding optimal control problem with value function $u^{n+1}$. Their optimal (closed-loop) control is then $(t,x)\to -D_pH(x,\nabla u^{n+1}(t,x))$. When all players actually implement this control the population density evolves in time according to \eqref{MFG2FP}-(ii). We assume that the players observe the resulting evolution of the population density $(m^{n+1}(t))_{t\in [0,T]}$. At the end of stage $n+1$ the players update their guess by computing the {\it average} of their observation up to time $n+1$.

In order to show the convergence of the Fictitious Play, we assume  that the MFG is potential, i.e. there are potential functions $F,G : \mes (\states) \to \R$ such that
\be\label{potpot}
f(x,m)=\frac{\delta F}{\delta m}(m,x) \qquad {\rm and}\qquad g(x,m)=\frac{\delta G}{\delta m}(m,x).
\ee
We also assume that $m_0$ is smooth and positive.

\begin{Theorem}\label{MainTheoremSecondOreder} Under the assumptions \eqref{hypm0}, \eqref{hyp:unifCv}, \eqref{hyp:Hcondsupp}, \eqref{regucondf} and \eqref{potpot}, the family $\{ (u^n , m^n) \}_{n \in \N}$ is uniformly continuous and any cluster point is a solution to the  second order MFG \eqref{MFG2}. 

If, in addiction, the monotonicity condition \eqref{mono} holds, then the whole sequence $\{ (u^n , m^n) \}_{n \in \N}$ converges to the unique solution of  \eqref{MFG2}. 
\end{Theorem}

The key remark to prove Theorem \ref{MainTheoremSecondOreder} is that the game itself has a {\it potential}. Given $m\in C^0([0,T]\times\states)$ and $w\in C^0([0,T]\times\states)$ such that, in the sense of distribution,  
$$
\partial_t m-\Delta m+{\rm div}(w)=0\; {\rm in }\; (0,T)\times \states   \qquad m(0)=m_0,
$$
let 
$$\Phi( m , w ) = \int_{0}^{T} \int_{\states} m(t,x) H^{*}(x, - w(t,x)/ m(t,x) )  \d x \d t + \int_{0}^{T} F(m(t)) \d t + G(m(T)),$$
where $H^*$ is the  convex conjugate of $H$: 
$$H^*(x,q)= \sup_{p\in \R^d} \; \lg p,q\rg - H(x,p).$$
In the definition of  $\Phi$, we set by convention when $m=0$, 
$$
H^{*}(x, - w/ m )= \left\{\begin{array}{ll}
0 & {\rm if}\; w=0\\
+\infty & {\rm otherwise.}
\end{array}\right.
$$
For sake of simplicity, we often drop the integration and the variable $(t,x)$ to write the potential in a shorter form:
$$\Phi( m , w ) = \int_{0}^{T} \int_{\states} m H^{*}(x, - w/m)  + \int_{0}^{T} F(m(t)) \d t + G(m(T)).$$
It is explained in \cite{LL07mf} section 2.6 that $(u,m)$ is a solution to \eqref{MFG2} if and only if $(m,w)$ is a minimizer of $\Phi$ and $w= -mD_pH(\cdot, \nabla u)$. 
We show here that the same map can be used as a potential in the Fictitious Play: $\Phi$ (almost) decreases at each step of the Fictitious Play and the derivative of $\Phi$ does not vary too much at each step. Then the proof of \cite{MondrerShapley2} applies. 

%%%%%%%%%%%%%%%%%%%%%%%%%%%%%%%%
\subsection{Proof of the convergence}

Before starting the proof of Theorem \ref{MainTheoremSecondOreder}, let us fix some notations. First we set 
\be\label{def.wn}
w^{n}(t,x) = - m^{n} (t,x) D_p H (x,\nabla u^{n}(t,x))\; {\rm and}\; \bar{w}^{n}(t,x) = \frac{1}{n} \sum_{k=1}^{n} w^{k}(t,x).
\ee
Since the Fokker-Planck equation is linear we have :
\be\label{e.FPbarm}
\partial_t \bar{m}^{n+1} - \Delta \bar{m}^{n+1}  + \text{div} (\bar w^{n+1}) =0 , \quad t\in [0,T], \qquad \bar{m}^{n+1}(0)=m_0.
\ee
Recall that $H^*$ is the  convex conjugate of $H$: 
$$H^*(x,q)= \sup_{p\in\R^d} \; \lg p,q\rg - H(x,p).$$
We define $\hat p (x,q)$ as the minimum in the above right-hand side:
\begin{equation}\label{def.hatp}
H^*(x,p)= \lg \hat p(x,q),q\rg - H(x,\hat p(x,q)). 
\end{equation}
Note that $\hat p$ is characterized by $q= D_p H(x,\hat p (x,q))$. The uniqueness comes from the fact that $H$ satisfies $D_{pp} H \geq \frac{1}{C} I_d$, which yields that $D_p H(x,\cdot)$ is one-to-one. We note for later use that
$$mH^{*} (x, - \frac{q}{m}) = \sup_{p\in \R^d} \;  - \lg p,q\rg -mH(x,p).$$
Next we state a standard result on uniformly convex functions, the proof of which is postponed: 
\begin{Lemma}\label{FenchelConjugate}
Under assumption \eqref{hyp:unifCv}, we have for any $x\in \states$, $p,q\in \R^d$:
$$H(x,p) + H^*(x,q) - \lg p,q\rg \geq \frac{1}{2\bar C} \leftnorm q-D_p H(x,p) \rightnorm^2 $$
\end{Lemma}

The following Lemma explains that $\Phi$  is ``almost decreasing" along the sequence $(\bar m^n, \bar w^n)$. 
\begin{Lemma}\label{lem.DecreasePhi} There exists a constant $C>0$ such that, for any $n\in\N^*$,  
\begin{equation}
\Phi( \bar m^{n+1}, \bar w^{n+1} ) - \Phi(\bar m^n \bar w^n ) \leq  - \frac{1}{C} \frac{a_n}{n} + \frac{C}{n^2},
\end{equation}
where $\ds a_{n} =  \int_{0}^{T} \mint \bar m^{n+1} \labs  \bar w^{n+1} / \bar m^{n+1} - w^{n+1}/m^{n+1} \labs ^2 $. 
\end{Lemma}

Throughout the proofs, $C$ denotes a constant which depends on the data of the problem only (i.e., on $H$, $f$, $g$ and $m_0$) and might change from line to line. We systematically use the fact that, as $f$ and $g$ admit $F$ and $G$ as a potential and are globally Lipschitz continuous, there exists a constant $C >0$ such that, for any $m,m'\in \mes (\states)$ and $s\in [0,1]$, 
$$ \leftnorm F(m + s(m' - m)) - F(m) - s \mint f(x,m) \d (m'-m)(x) \rightnorm < C \abs s \abs ^2 ,$$
$$ \leftnorm G(m + s(m' - m)) - G(m) - s \mint g(x,m) \d (m'-m)(x) \rightnorm < C \abs s \abs ^2.$$

\begin{proof}[Proof of Lemma \ref{lem.DecreasePhi}]
We have
$$\Phi(\bar{m}^{n+1} , \bar{w}^{n+1}) = \Phi(\bar{m}^{n} , \bar{w}^{n}) +  A + B ,$$
where
\begin{equation}
\ds A = \int_0^T\int_{\states} \bar m^{n+1}H^*( - \bar w^{n+1}/\bar m^{n+1})- \bar m^{n}H^*( - \bar w^{n}/\bar m^{n})
\end{equation}
\begin{equation}
\ds B = \int_{0}^{T} \big( F(\bar m^{n+1} (t)) - F(\bar m^n (t)) \big) \d t + \big( G(\bar m^{n+1} (T)) - G(\bar m^n (T)) \big).
\end{equation}
Since $F$ is $C^1$ with respect to $m$ with derivative $f$, we have
$$
B \leq \int_{0}^{T} \mint f(x,\bar m^n (t) ) (\bar m^{n+1} - \bar m^{n} )+ \mint g(x,\bar m^n(T)) (\bar m^{n+1}- \bar m^{n}) + \frac{C}{n^2} .
$$
As $\ds \bar m^{n+1} - \bar m^{n}= \frac{1}{n+1}( m^{n+1} - \bar m^{n+1})$,   we find  after rearranging: 
$$
B \leq \frac{1}{n+1}\int_{0}^{T} \mint f(x,\bar m^n (t) )  ( m^{n+1} - \bar m^{n+1}) + \frac{1}{n+1} \mint g(x,\bar m^n(T)) ( m^{n+1}(T)- \bar m^{n+1}(T))  + \frac{C}{n^2} .
$$
Using now the equation satisfied by $u^{n+1}$ we get
\begin{equation*}
\begin{split}
B &\leq \frac{1}{n+1}\int_{0}^{T} \mint \big( -\partial_t u^{n+1} - \Delta u^{n+1}  +H(x,\nabla u^{n+1})  \big)  ( m^{n+1} - \bar m^{n+1}) \\
& \qquad \qquad +\frac{1}{n+1} \mint g(x,\bar m^n(T)) ( m^{n+1}(T)- \bar m^{n+1}(T))  + \frac{C}{n^2} \\
&\leq \frac{1}{n+1}\int_{0}^{T} \mint \big(\partial_t  ( m^{n+1} - \bar m^{n+1}) - \Delta  (m^{n+1} - \bar m^{n+1}) \big) u^{n+1}\\
& \qquad\qquad   + \frac{1}{n+1}\int_{0}^{T} \mint  H(x,\nabla u^{n+1})  ( m^{n+1} - \bar m^{n+1}) +\frac{C}{n^2},  
\end{split}
\end{equation*}
where we have integrated by parts in the second inequality. Using now the equation satisfied by $m^{n+1} - \bar m^{n+1}$ and integrating again by parts, we obtain
$$
B \leq \frac{1}{n+1}\int_{0}^{T} \mint  \lg w^{n+1} - \bar{w}^{n+1} , \nabla u^{n+1} \rg  + H(x,\nabla u^{n+1}) ( m^{n+1} - \bar m^{n+1}) +\frac{C}{n^2}.
$$
Note that  by Lemma \ref{FenchelConjugate},  
\begin{equation*}
\begin{split}
-\lg \bar{w}^{n+1} , \nabla u^{n+1} \rg  - H(x,\nabla u^{n+1})  \bar m^{n+1} 
\leq & \;  \bar m^{n+1} H^* ( - \bar w^{n+1} / \bar m^{n+1})\\
& \qquad  - \frac{1}{2\bar C } \bar m^{n+1} \labs  \bar w^{n+1} / \bar m^{n+1} - w^{n+1}/m^{n+1} \labs ^2 
\end{split}
\end{equation*}
while, by the definition of $w^{n+1}$, 
$$
\lg w^{n+1} , \nabla u^{n+1} \rg  + H(x,\nabla u^{n+1}) m^{n+1} = -  m^{n+1} H^* (- w^{n+1} / m^{n+1}).
$$
Therefore
\begin{equation}\label{e.iabzohs}
\begin{split}
B &\leq \frac{1}{n+1}\int_{0}^{T} \mint \bar m^{n+1} H^* ( - \bar w^{n+1} / \bar m^{n+1}) -  m^{n+1} H^* (- w^{n+1} / m^{n+1})\\
&\qquad \qquad - \frac{1}{2\bar C n}\int_{0}^{T} \mint \bar m^{n+1} \labs  \bar w^{n+1} / \bar m^{n+1} - w^{n+1}/m^{n+1} \labs ^2  +\frac{C}{n^2}.
\end{split}
\end{equation}
On the other hand, recalling the definition of $\hat p$ in \eqref{def.hatp} and setting $\bar p^{n+1} = \hat p (\cdot , - \bar w^{n+1} / \bar m^{n+1})$, we can estimate  $A$ as follows:
\begin{equation}\label{Aexp}
\begin{split}
A &\leq \int_{0}^{T} \mint  - \lg \bar p ^{n+1} , \bar w ^{n+1} \rg -\bar m^{n+1} H(x, \bar p ^{n+1} ) + \lg \bar p ^{n+1}, \bar w ^{n} \rg + \bar m^{n} H(x, \bar p ^{n+1} ) \\
&= \frac{1}{n+1} \int_{0}^{T} \mint  \lg \bar p ^{n+1} , \bar w ^{n+1} \rg + \bar m^{n+1} H(x, \bar p ^{n+1} ) - \lg \bar p ^{n+1}, w ^{n+1} \rg - m^{n+1} H(x, \bar p ^{n+1} ) \\
&\leq \frac{1}{n+1} \int_{0}^{T} \mint m^{n+1} H^*(- w^{n+1} / m^{n+1} ) - \bar m^{n+1} H^*( - \bar w^{n+1} / \bar m^{n+1} ).
\end{split}
\end{equation}
 Putting together \eqref{e.iabzohs} and \eqref{Aexp} we find:
%\begin{equation}
$$
\Phi( \bar m^{n+1}, \bar w^{n+1}) - \Phi(  \bar m^n, \bar w^n) \leq  - \frac{1}{2\bar C} \frac{a_n}{n} + \frac{C}{n^2}
$$
%\end{equation}
where $\ds a_{n} =  \int_{0}^{T} \mint \bar m^{n+1} \labs  \bar w^{n+1} / \bar m^{n+1} - w^{n+1}/m^{n+1} \labs ^2 $. 
\end{proof}

In order to proceed, let us recall some basic estimates on the system \eqref{MFG2FP}, the proof of which is postponed:
\begin{Lemma}\label{lem.regumw} For any $\alpha\in (0,1/2)$ there exist a constant $C > 0$ such that
for any $n\in \N^*$ 
$$
\|u^n\|_{C^{1+\alpha/2, 2+\alpha}} + \|m^n\|_{C^{1+\alpha/2, 2+\alpha}}\leq C, \qquad 
m^n\geq 1/C,
$$
where $C^{1+\alpha/2, 2+\alpha}$ is the usual H\"{o}lder space on $[0,T]\times\T^d$. 
%\abs \nabla u^n (t,x) ) \abs + |\nabla^2 u^n(t,x)|+
%|\partial_t \nabla u^n(t,x)| \leq C , \quad \frac{1}{A} \leq m^n (t,x) \leq C.$$
\end{Lemma}

As a consequence, the $u^n$, the $m^n$ and the $w^n$ do not vary too much between two consecutive steps: 
\begin{Lemma}\label{lem.un+1-un} There exists a constant $C>0$ such that 
$$
\|u^{n+1}-u^n\|_\infty+ \|\nabla u^{n+1}-\nabla u^n\|_\infty+ \|m^{n+1}-m^n\|_\infty+ \|w^{n+1}-w^n\|_\infty\leq \frac{C}{n}.
$$
\end{Lemma}

\begin{proof} As $\bar m^{n}-\bar m^{n-1}= ((n-1)\bar m^{n-1}+m^{n})/n$, where the $m^n$ (and thus the $\bar m^n$) are uniformly bounded thanks to Lemma \ref{lem.regumw}, we have by Lipschitz continuity of $f$ and $g$ that 
\begin{equation}\label{e.un+1-un3}
\sup_{t\in [0,T]} \left\| f(\cdot, \bar m^{n+1}(t))-f(\cdot, \bar m^{n}(t))\right\|_\infty + 
\left\| g(\cdot, \bar m^{n+1}(T))-g(\cdot, \bar m^{n}(T))\right\|_\infty \leq \frac{C}{n}.
\end{equation}
Thus, by comparison for the solution of the Hamilton-Jacobi equation, we get 
\begin{equation}\label{e.un+1-un2}
\|u^{n+1}-u^n\|_\infty \leq \frac{C}{n}.
\end{equation}
Let us set $z:= u^{n+1}-u^n$. Then $z$ satisfies 
%\begin{equation*}
%\begin{split}
%-\partial_t z-\Delta z - C|\nabla z|  &  \leq  \; 
%-\partial_t z-\Delta z +H(x,\nabla u^n+\nabla z)-H(x,\nabla u^n) \\
%&  \leq \; f(x, \bar m^n(t))-f(x, \bar m^{n-1}(t)).
%\end{split}
%\end{equation*}
$$
-\partial_t z-\Delta z +H(x,\nabla u^n+\nabla z)-H(x,\nabla u^n) \;  = \; f(x, \bar m^n(t))-f(x, \bar m^{n-1}(t)).
$$
Multiplying by $z$ and integrating over $[0,T]\times \T^d$ we find by \eqref{e.un+1-un3} and \eqref{e.un+1-un2}:
$$
-\left[\mint \frac{z^2}{2}\right]_0^T + \int_0^T \mint |\nabla z|^2  + z (H(x,\nabla u^n+\nabla z)-H(x,\nabla u^n)) \leq \frac{C}{n^2}.
$$
Then we use the uniform bound on the $\nabla u^n$ given by Lemma \ref{lem.regumw} as well as  \eqref{e.un+1-un2} to get 
$$
 \int_0^T \mint (|\nabla z|^2  - \frac{C}{n}  |\nabla z|) \leq \frac{C}{n^2}.
$$
Thus 
$$
 \int_0^T \mint |\nabla z|^2\leq \frac{C}{n^2},
 $$
 which implies that $\|\nabla z\|_\infty \leq C/n$ since $\|\nabla^2z\|_\infty+ \|\partial_t \nabla z\|_\infty \leq C$ by Lemma \ref{lem.regumw}.

 We argue in a similar way for $\mu:=m^{n+1}-m^n$: $\mu$ satisfies
 $$
 \partial_t \mu -\Delta \mu - \text{div} ( \mu D_pH(x,Du^{n+1}))- \text{div}  (R)=0, 
 $$
 where we have set $R= m^n\left(D_pH(x,\nabla u^{n+1})-D_pH(x,\nabla u^{n})\right)$. As $\|R\|_\infty\leq C/n$ by the previous step, we get the bound on  $ \|m^{n+1}-m^n\|_\infty\leq C/n$ by standard parabolic estimates. This implies the bound on $\|w^{n+1}-w^n\|_\infty$ by the definition of the $w^n$. 
\end{proof}

Combining Lemma \ref{lem.regumw} with Lemma \ref{lem.un+1-un} we immediately obtain that the sequence $(a_n)$ defined in Lemma \ref{lem.DecreasePhi} is slowly varying in time: 
\begin{Corollary}\label{coro.bn+1-bn} There exists a constant $C>0$ such that, for any $n\in \N^*$, 
$$
\left|a_{n+1}-a_n\right|\leq \frac{C}{n} .
$$
\end{Corollary}

\begin{proof}[Proof of Theorem \ref{MainTheoremSecondOreder}] From Lemma \ref{lem.DecreasePhi}, we have for any $n\in\N^*$,  
$$
\Phi(\bar m^{n+1}, \bar w^{n+1} ) - \Phi( \bar m^n,  \bar w^n) \leq  - \frac{1}{C} \frac{a_n}{n} + \frac{C}{n^2}
$$
where $\ds a_{n} =  \int_{0}^{T} \mint \bar m^{n+1} \labs  \bar w^{n+1} / \bar m^{n+1} - w^{n+1}/m^{n+1} \labs ^2 $. 

Since the potential $\Phi$ is bounded from below the above inequality implies that 
$$
\sum_{n\geq1} a_n / n < +\infty.
$$ 
%Then Lemma \ref{Convtozero} below yields:
%\begin{equation*}
%\lim_{N \to \infty} \frac{1}{N} \sum_{n=1}^{N} a_n = 0.
%\end{equation*}
From Corollary \ref{coro.bn+1-bn}, we also have, for any $n\in \N^*$, 
$$
\left|a_{n+1}-a_n\right|\leq \frac{C}{n} .
$$
Then Lemma \ref{Convtozero} below implies that $\lim_{n \to \infty} a_n = 0$. 

In particular we have, by Lemma \ref{lem.regumw}:
$$ \lim_{n \to \infty} \int_{0}^{T} \mint \labs \bar w^{n} / \bar m^{n} - w^{n}/m^{n} \labs ^2 \leq C \lim_{n \to \infty} \int_{0}^{T} \mint \bar m^{n} \labs \bar w^{n} / \bar m^{n} - w^{n}/m^{n} \labs ^2 =0.$$
This implies that the sequence $\{ \bar w^{n} / \bar m^{n} - w^{n}/m^{n} \}_{ n \in \N}$---which is uniformly continuous from Lemma \ref{lem.regumw}---uniformly converges to $0$ on $[0,T] \times \states$.

Recall that, by  Lemma \ref{lem.regumw}, the sequence $\{ (u ^{n+1} , m^{n}, \bar m ^n , \bar w^n) \}_{n \in \N}$ is pre-compact for the uniform convergence. Let $(u , m, \bar m , \bar w)$ be a cluster point of the sequence $\{ (u ^{n+1} , m^{n}, \bar m ^n , \bar w^n) \}_{n \in \N}$. Our aim is to show that $(u,m)$ is a solution to the MFG system \eqref{MFG2}, that $\bar m=m$ and that $\bar w= -mD_pH(\cdot, \nabla u)$.  

Let $n_i \in \N , i \in \N$ be a subsequence such that $(u^{n_i +1}, m^{n_i}, \bar m^{n_i}, w^{n_i})$ uniformly converges to $(u, m, \bar m, \bar w)$. By the estimates in Lemma \ref{lem.regumw}, we have $D_p H(x,\nabla u^{n_j})$  converges uniformly to 
$D_p H(x,\nabla u)$, so that by \eqref{def.wn} and the fact that the sequence $\{ \bar w^{n} / \bar m^{n} - w^{n}/m^{n} \}_{ n \in \N}$ converges to $0$, 
\begin{equation}\label{e.ljqdsspdjc}  
-D_p H(x,\nabla u) =  \frac{w}{m} = \frac{\bar w}{ \bar m}. 
\end{equation}
We now pass to the limit in \eqref{MFG2FP} (in the viscosity sense for the Hamilton-Jacobi equation and in the sense of distribution for the Fokker-Planck equation) to get
\begin{equation}\label{MFG3}
\begin{split}
(i) \qquad &-\partial_t u - \Delta u  +H(x,\nabla u (t,x)) =f(x,\bar{m}(t)), \quad (t,x)\in  [0,T]\times \states  \\
(ii) \qquad &\partial_t  m - \Delta  m  - \text{div} ( mD_pH(x,\nabla u)) =0 , \quad (t,x)\in  [0,T]\times \states  \\
& m(0)=m_0 , \; u(x,T)=g(x, \bar{m} (T)), \quad  x\in \states.
\end{split}
\end{equation}
Letting $n\to +\infty$ in \eqref{e.FPbarm} we also have 
$$
\partial_t  \bar m - \Delta \bar  m  + \text{div} ( \bar w) =0 , \quad t\in [0,T] , \qquad \bar m(0)=m_0.
$$
By \eqref{e.ljqdsspdjc}, this means that $m$ and $\bar m$ are both solutions to the same Fokker-Planck equation. Thus they are equal and $(u,m)$ is a solution to the MFG system. 

If \eqref{mono} holds, then the MFG system has a unique solution $(u,m)$, so that the compact sequence $\{(u^n,m^n)\}$ has a unique accumulation point  $(u,m)$ and thus converges to $(u,m)$. 
\end{proof}

In the proof of Theorem \ref{MainTheoremSecondOreder}, we have used the following Lemma, which can be found in \cite{MondrerShapley2}. 

\begin{Lemma}\label{Convtozero}
Consider a sequence of positive real numbers $\{ a_n \}_{n \in \N}$ such that $\sum_{n=1}^{\infty} a_n / n < +\infty$. Then we have
$$\lim_{N \to \infty} \frac{1}{N} \sum_{n=1}^{N} a_n = 0.$$
In addition, if there is a constant $\bar C >0 $ such that $\abs a_n - a_{n+1} \abs < \frac{\bar C}{n}$ then $\lim_{n \to \infty} a_n = 0$
\end{Lemma}
\begin{proof} We reproduce the proof of \cite{MondrerShapley2} for the sake of completeness. 
For every $k \in \N$ define $b_k = \sum_{n=k}^{\infty} a_n / n $. Since $\sum_{n=1}^{\infty} a_n / n < +\infty$ we have $\lim_{k \to \infty} b_k = 0$.  So we have:
$$\lim_{N \to \infty} \frac{1}{N} \sum_{k=1}^{N} b_k = 0,$$
which yields the first result since:
$$\sum_{n=1}^{N} a_n \leq \sum_{k=1}^{N} b_k.$$
For the second result, consider $\epsilon >0$. We know that for every $\lambda > 0$ we have:
$$\lim_{N \to \infty} \frac{1}{N} + \frac{1}{N+1} + \cdots + \frac{1}{[(1+\lambda)N]} = \log (1 + \lambda ), $$
where $[a]$ denotes the integer part of the real number $a$. 
So if $\lambda_\epsilon >0$ is so small that $\log(1+ \lambda_\epsilon ) < \frac{\epsilon}{2\bar C}$, then there exist $N_\epsilon \in \N$ so large that for $N \geq N_\epsilon$ we have
\be\label{ljhzqsldj}
\frac{1}{N} + \frac{1}{N+1} + \cdots + \frac{1}{[(1+\lambda_\epsilon)N]} < \frac{\epsilon}{2\bar C} .
\ee
Let $N \geq N_\epsilon$. Assume for a while that $a_N>\ep$. As $|a_{k+1}-a_k|\leq \bar C/k$, \eqref{ljhzqsldj} implies that $a_k >\frac{\epsilon}{2}$ for $N \leq k \leq [N(1+ \lambda_\epsilon)]$. Thus
$$\frac{1}{[N(1+\lambda_\epsilon)]} \sum_{k=1}^{[N(1+\lambda_\epsilon)]} a_k \geq \frac{\lambda_\epsilon}{1+\lambda_\epsilon} \frac{\epsilon}{2}.$$
Since the average $N^{-1}\sum_{k=1}^N a_k$ converges to zero, the above inequality cannot hold for $N$ large enough. This implies that $a_N\leq \ep$ for $N$ sufficiently large, so that $(a_k)$ converges to $0$.
\end{proof}

\begin{proof}[Proof of Lemma \ref{FenchelConjugate}.] For simplicity of notation,we omit the $x$ dependence in the various quantities. As by assumption \eqref{hyp:unifCv} we have $\frac{1}{C} I_d \leq D ^2 _{pp}H \leq C I_d$, $H^*$ is differentiable with respect to $q$ and the following inequality holds: for any $q_1,q_2\in \R^d$, 
$$
\lg D_qH^*(q_1)-D_qH^*(q_2), q_1-q_2\rg \geq \frac{1}{\bar C} |q_1-q_2|^2. 
$$
Let us fix $p,q\in \R^d$ and let $\hat q\in \R^d$ be the maximum in
$$
\max_{q'\in \R^d} \lg q',p\rg -H^*(q') = H(p).
$$
Recall that $p=D_qH^*(\hat q)$ and thus $\hat q= D_pH(p)$. 
Then 
$$
\begin{array}{rl}
\ds H(p) + H^*(q) - \lg p,q\rg 
&= \ds H^*(q) - H^*(\hat q) - \lg q-\hat q ,p \rg\\
& = \ds \int_0^1 \lg D_qH^*((1-t)\hat q+tq) - D_qH^*(\hat q),q-\hat q\rg dt \\
& = \ds  \int_0^1 \frac{1}{t} \lg D_qH^*((1-t)\hat q+tq) - D_qH^*(\hat q),((1-t)\hat q+tq)- \hat q\rg dt \\
%& \\
&\ds  \geq \int_0^1 t\ \frac{1}{\bar C} |\hat q-q|^2= \frac{1}{2\bar C}  |D_pH(p)-q|^2.
\end{array}
$$
\end{proof}

\begin{proof}[Proof of Lemma \ref{lem.regumw}.] Given $\bar m^n\in C^0([0,T], \mes(\T^d))$, the solution $u^{n+1}$ is uniformly Lipschitz continuous. Hence any weak solution to the Fokker-Planck  equation is uniformly H\"{o}lder continuous in $C^0([0,T], \mes(\T^d))$. This shows that the right-hand side of the Hamilton-Jacobi equation is uniformly H\"{o}lder continuous; then the Schauder estimate provide the bound in $C^{1+\alpha/2, 2+\alpha}$ for $\alpha\in (0,1/2)$. Plugging this estimate into the Fokker-Planck equation and using again the Schauder estimates gives the bounds in $C^{1+\alpha/2, 2+\alpha}$ on the the $m^n$. The bound from below for the $m^n$ comes from the strong maximum principle.
\end{proof}

%%%%%%%%%%%%%%%%%%%%%%%%%%%%%%%%%%%%%%%
%%%%%%%%%%%%%%%%%%%%%%%%%%%%%%%%%%%%%%
\section{The Fictitious Play for first order MFG systems}\label{sec.MFG1}

We now consider the first order order MFG system: 
\begin{equation}\label{MFG}
\left\{\begin{split}
(i) \qquad&-\partial_t u  +H(x,\nabla u(t,x)) =f(x,m(t)), \quad (t,x)\in  [0,T]\times \states \\
(ii) \qquad&\partial_t m  + \text{div} (-mD_p H(x,\nabla u(t,x)) ) =0 , \quad (t,x)\in  [0,T]\times \states \\
&m(0)=m_0 , \; u(x,T)=g(x, m(T)), \quad  x\in \states
\end{split}\right.
\end{equation}
In contrast with second order MFG systems, we cannot expect the existence of classical solutions: namely both the Hamilton-Jacobi equation and the Fokker-Planck equation have to be understood in a generalized sense. In particular, the solutions of the Fictitious Play are not smooth enough to justify the various computations of section \ref{sec.MFG2}. For this reason we introduce another method---based on another potential---, which also has the interest that it can be adapted to a finite population of players.  

Let us start by recalling the notion of solution for \eqref{MFG}. Following \cite{LL07mf}, we say that the pair $(u,m)$ is a solution to the MFG system \eqref{MFG} if $u$ is a Lipschitz continuous  viscosity solution to \eqref{MFG}-(i) while $m\in L^\infty((0,T)\times \T^d)$ is a solution of \eqref{MFG}-(ii) in the sense of distribution. 

Under our standing assumptions \eqref{hypm0}, \eqref{hyp:unifCv}, \eqref{hyp:Hcondsupp}, \eqref{regucondf},  there  exists at least one solution $(u,m)$ to the mean field game system \eqref{MFG}. If furthermore \eqref{mono} holds, then the solution is unique  (see \cite{LL07mf} and Theorem 5.1  in \cite{Cdga13}).

%%%%%%%%%%%%%%%%%%%%%%%%%%%%%%%%%%
 \subsection{The learning rule and the potential}
%\subsubsection{The learning rule}

The learning rule is basically the same as for second order MFG systems: given a smooth initial guess $m^0:[0,T]\times \T^d\to \R$, we define by induction sequences $u^n,m^n:[0,T]\times \T^d\to \R$ {\it heuristically} given by:
\begin{equation}\label{MFG1FP}
\begin{split}
(i) \qquad &-\partial_t u^{n+1}  +H(x,\nabla u^{n+1}(t,x)) =f(x,\bar{m}^{n}(t)), \quad (t,x)\in  [0,T]\times \states  \\
(ii) \qquad &\partial_t m^{n+1} + \text{div} ( - m^{n+1}D_p H (x,\nabla u^{n+1})) =0 , \quad (t,x)\in  [0,T]\times \states  \\
&m^{n+1}(0)=m_0 , \; u^{n+1}(x,T)=g(x, \bar{m} ^n (T)), \quad  x\in \states
\end{split}
\end{equation}
where $\bar{m}^{n} (t,x) = \frac{1}{n} \sum_{k=1}^{n} {m}^{k}(t,x)$. If equation \eqref{MFG1FP}-(i) is easy to interpret, the meaning of  \eqref{MFG1FP}-(ii) would be  more challenging and, actually, would make little sense for a finite population. For this reason we are going to rewrite the problem in a completely different way, as a problem on the space of curves. 

Let us fix the notation. Let $\Gamma = C^0([0,T], \states)$ be the set of curves. It is endowed with usual topology of the uniform convergence and we denote by $\mathcal{B}(\Gamma)$ the associated $\sigma-$field. We define $\mes(\Gamma)$ as the set of Borel probability measures on $\mathcal{B}(\Gamma)$. We view $\Gamma$ and $\mes(\Gamma)$ as the set of pure and mixed strategies for the players. For any $t \in [0,T]$ the evaluation map $e_t : \Gamma \to \states$, defined by:
$$e_t (\gamma) = \gamma (t),\quad \forall \gamma \in \Gamma$$
is continuous and thus measurable. For any $\eta \in  \mes(\Gamma)$ we define $m^\eta(t) = e_t \sharp \eta$ as the push forward of the measure $\eta$ to $\states$ i.e.
$$m^\eta(t)(A) = \eta (\{ \gamma \in \Gamma \mid \gamma(t)\in A \})$$
for any measurable set $A\subset \states$. We denote by $\mes_0(\Gamma)$ the set of probability measures on $\Gamma$ such that $e_0\sharp \eta=m_0$. Note that $\mes_0(\Gamma)$ is the set of strategies compatible with the initial density $m_0$. 

Given an initial time $t\in [0,T]$ and an initial position $x$, it is convenient to define the cost of a path $\gamma\in C^0([t,T],\states)$ payed by a small player starting from that position when the repartition of strategies of the other players is $\eta$. It is given by 
$$
J(t,x,\gamma, \eta):= \left\{ \begin{array}{ll}
\ds \int_t^T L(\gamma(s),\dot \gamma(s))+ f(\gamma(s),m^\eta(s)) \d s + g(\gamma(T),m^\eta(T)) & {\rm if }\; \gamma\in H^1([t,T],\states)\\
+\infty & {\rm otherwise}.
\end{array}\right.
$$
where $L(x,v):= H^*(x,-v)$ and $H^*$ is the Fenchel conjugate of $H$ with respect to the last variable. If $t=0$, we simply abbreviate $J(x,\gamma, \eta):= J(0,x,\gamma, \eta)$. We note for later use that $J(t,x,\cdot,\eta)$ is lower semi-continuous on $\Gamma$. 

We now define the Fictitious Play.  We start with an initial configuration $\eta^0\in \mes_0(\Gamma)$ (the belief before the first step of a typical player on the actions of the other players). We now build by induction the sequences $(\theta^n)$ and $(\eta^n)$ of $\mes(\Gamma)$, $\eta^n$ being interpreted as the belief at the end of stage $n$ of a typical player on the actions of the other agents and $\theta^{n+1}$ the repartition of strategies of the players when they play optimally in the game against $\eta^n$. More precisely, for any $x\in \states$, let $\bar \gamma_x^{n+1}\in H^1([0,T], \states)$ be an optimal solution to 
$$
\inf_{\gamma\in H^1, \ \gamma(0)=x } J(x,\gamma, \eta^n). 
$$
In view of our coercivity assumptions on $H$ and the definition of  $L$, the optimum is known to exist. Moreover, by the measurable selection theorem we can (and will) assume that the map $x\to \bar \gamma_x^{n+1}$ is Borel measurable. We then consider the measure $\theta^{n+1} \in P_0(\Gamma)$ defined by 
$$
\theta^{n+1}:= \bar \gamma^{n+1}_{\cdot} \sharp m_0\qquad \forall t\in[0,T]
$$
and set 
\begin{equation}\label{defetak}
\eta^{n+1}:= \frac{1}{n+1} \sum_{k=1}^{n+1} \theta^k = \eta^n + \frac{1}{n+1}(\theta^{n+1}-\eta^n).
\end{equation}
%We also set
%$$
%m^{n+1}(t)= e_t\sharp \theta^{n+1}.
%$$
%With the notation of section \eqref{sec.MFG2}, we have 
%$$
%\bar m^{n+1}(t)= e_t\sharp \eta^{n+1}, \qquad u^{n+1}(t,x)= J(t,x,\bar \gamma_x^{n+1},  \eta^n).
%$$

As in section \ref{sec.MFG2}, we assume that our MFG is potential, i.e., that there exists of potential functions $F,G : \mes (\states) \to \R$ such that:
\be\label{potpot2}
f(x,m)= \frac{\delta F}{\delta m}(x,m), \quad g(x,m)= \frac{\delta G}{\delta m}(x,m).
\ee

Here is our main convergence result.
\begin{Theorem}\label{th.mainOrdre1} Assume that \eqref{hypm0}, \eqref{hyp:unifCv}, \eqref{hyp:Hcondsupp}, \eqref{regucondf} and \eqref{potpot2} hold. Then the sequences $(\eta^n,\theta^n)$ is pre-compact in $\mes(\Gamma)\times \mes(\Gamma)$ and any cluster point $(\bar \eta, \bar \theta)$ satisfies the following: $\bar \theta =\bar \eta$ and, if we set 
\be\label{e.defbarmbaru}
\bar m(t):= e_t \sharp \bar \eta
, \quad \bar u(t,x)= \inf_{\gamma\in H^1, \ \gamma(t)=x} J(t,x,\gamma, \bar \eta),
\ee
then the pair $(\bar u,\bar m)$ is a solution to the MFG system \eqref{MFG}. 

If furthermore \eqref{mono} holds, then the entire sequence $(\eta^n,\theta^n)$ converges. 
\end{Theorem}

The proof of Theorem \ref{th.mainOrdre1} is postponed to the next subsection. 
As for the second order problem, the key idea is that our MFG system has a potential. However, in contrast with the second order case, the potential is now written on the space of probability on curves and reads, for $\eta \in \mes(\Gamma)$, 
\begin{equation}\label{defPhi2}
\Phi(\eta):= \int_\Gamma \int_0^T L(\gamma(t),\dot \gamma(t))\ dt \d \eta(\gamma) + \int_0^T F( e_t\sharp \eta)\ dt + G(e_T\sharp \eta).
\end{equation}
Note that $\Phi(\eta)$ is well-defined and belongs to $(-\infty, +\infty]$. The potential defined above is reminiscent of \cite{CCN} or \cite{c1}. For instance, in \cite{c1}---but for MFG system with a  local dependence and under the monotonicity condition \eqref{mono}---it is proved that the MFG equilibrium can be found as a global minimum of $\Phi$. We will show in the proof of Theorem \ref{th.mainOrdre1} that the limit measure $\bar \eta$ is characterized by the optimality condition
$$
\frac{\delta \Phi}{\delta m}(\bar \eta)(\bar \eta)\leq \frac{\delta \Phi}{\delta m}(\bar \eta)(\theta)\qquad \forall \theta \in \mes(\Gamma). 
$$

Before proving that $\Phi$ is a potential for the game, let us start with preliminary remarks. The first one 
explains that the optimal curves are uniformly Lipschitz continuous. 
\begin{Lemma}\label{LipschitzBound} There exists a constant $C>0$ such that, for any $x\in \states$ and any $n\geq 0$,  
\begin{equation}\label{e.lhqblsd}
\|\dot{\overline \gamma}_x^{n+1}\|_\infty \leq C.
\end{equation}
 In particular, the sequences $(\eta^n)$ and $(\theta^n)$ are tight and
$$
\dk(e_t\sharp \eta^{n+1},e_{t'}\sharp \eta^{n+1})\leq C |t-t'|\qquad \forall t,t'\in [0,T].
$$
\end{Lemma}

\begin{proof} 
Under our assumption on $H$, $f$ and $g$, it is known that the $(u^n)$ are uniformly Lipschitz continuous (see, for instance, the appendix of \cite{Cdga13}). As a byproduct the optimal solutions are also uniformly Lipschitz continuous thanks to the classical link between the derivative of the value function and the optimal trajectories (Theorem 6.4.8 of \cite{CannarsaSinestrari}): this is  \eqref{e.lhqblsd}. The rest of the proof is a straightforward consequence of \eqref{e.lhqblsd}. 
\end{proof}

Next we compute the derivative of $\Phi$ with respect to the measure $\eta$. Let us point out that, since $\Phi$ is not continuous and can take the value $+\infty$, the derivative, although defined by the formula \eqref{Potential}, has to be taken only at points and direction along which $\Phi$ is finite. This is in particular the case for the $\eta^n$ and the $\theta^n$. 

\begin{Lemma}\label{lem:repDPhi} For any $\eta,\eta'\in \mes (\Gamma)$ such that $\Phi(\eta), \Phi(\eta')<+\infty$, we have
$$
\frac{\delta \Phi}{\delta \eta} (\eta)(\eta'-\eta)\; =  \int_\Gamma J(\gamma(0), \gamma, \eta)\ d(\eta'-\eta)(\gamma). 
$$
\end{Lemma}

\begin{proof} This is a straightforward application of the definition of $\Phi$ in \eqref{defPhi2} and of the continuous derivability of $F$ and $G$. 
\end{proof}

By abuse of notation, we also define $\frac{\delta \Phi}{\delta \eta}(\eta)(\theta)$ for a positive Borel measure $\theta$ on $\Gamma$ by setting
$$
 \frac{\delta \Phi}{\delta \eta} (\eta)(\theta) = \int_\Gamma J(\gamma(0), \gamma, \eta) \d \theta(\gamma). 
$$
Note that, as $J$ is bounded below, the quantity $\frac{\delta \Phi}{\delta \eta} (\eta)(\theta)$ is well-defined and belongs to $(-\infty, +\infty]$.  

Next we translate the optimality property of $\bar \gamma_x^n$ to an optimality property of $\eta^n$. 
\begin{Lemma}\label{lem.optibargamma} For any $n\in \N^*$, 
$$
\frac{\delta \Phi}{\delta \eta} (\eta^n)(\theta^{n+1})= \int_{\states} J(x,\bar \gamma^{n+1}_x, \eta^{n})m_0(x)\d x
= \min_{\theta\in \mes_0(\Gamma)} \frac{\delta \Phi}{\delta \eta} (\eta^n)(\theta).
$$
\end{Lemma}

\begin{proof} The first equality is just the definition of $\theta^{n+1}$. It remains to check that, for any  $\theta\in \mes_0(\Gamma)$, 
$$
 \int_{\states} J(x,\bar \gamma^{n+1}_x, \eta^{n})m_0(x)\d x
\leq \int_{\Gamma} J(\gamma(0), \gamma, \eta^n) \d \theta(\gamma).
$$
As $m_0= e_0\sharp \theta$, we can disintegrate $\theta$ into $\theta= \mint \theta_x \d m_0(x)$, where $\theta_x\in \mes(\Gamma)$ with $\gamma(0)=x$ for $\theta_x-$a.e. $\gamma$. By optimality of $\bar \gamma^{n+1}_x$ we have, for $m_0-$a.e. $x\in \states$, 
$$
 J(x,\bar \gamma^{n+1}_x, \eta^n)   \leq \int_{\Gamma} J(x, \gamma, \eta^n)\ \d \theta_x(\gamma)
$$
and therefore, integrating with respect to $m_0$: 
%\begin{equation*}
%\begin{split}
$$
\mint  J(x,\bar \gamma^{n+1}_x, \eta^n) m_0(x) \d x\leq  \mint \int_{\Gamma} J(x, \gamma, \eta^n)\ \d \theta_x(\gamma) m_0(x)\d x
=  \int_{\Gamma} J(\gamma(0), \gamma, \eta^n) \d \theta(\gamma).
$$
%\end{split}
%\end{equation*}
\end{proof}

The next proposition states that the potential $\Phi$ is indeed almost decreasing along the sequence $(\eta^n)$. 

\begin{Proposition}\label{prop:Phi-Phi} There is a constant $C>0$ such that, for any $n\in \N^*$, we have 
\be\label{kjhbze}
\Phi(\eta^{n+1}) \leq \Phi(\eta^n) + \frac{1}{n+1} \frac{\delta \Phi}{\delta \eta} (\eta^n)(\theta^{n+1}-\eta^n)+ \frac{C}{(n+1)^2} 
\ee
where 
\begin{equation}
\begin{split}
\frac{\delta \Phi}{\delta \eta} (\eta^n)(\theta^{n+1}-\eta^n) = \int_\Gamma   J( \gamma(0), \gamma, \eta^n)\ \d(\theta^{n+1}-\eta^n)(\gamma) \leq 0. 
\end{split}
\end{equation}
\end{Proposition}

\begin{proof}  Recalling \eqref{defetak}, we have
\begin{equation}
\begin{split}
\Phi(\eta^{n+1})- \Phi(\eta^n) = & \int_0^1 \frac{\delta \Phi}{\delta \eta} ((1-s)\eta^n+s\eta^{n+1})(\eta^{n+1}-\eta^n) \d s \\
 = & \frac{1}{(n+1)} \int_0^1 \frac{\delta \Phi}{\delta \eta} ((1-s)\eta^n+s\eta^{n+1})(\theta^{n+1}-\eta^n) \d s.
%\leq & \ds \ds \frac{1}{(n+1)}  \frac{\delta \Phi}{\delta \eta} (\eta^n)(\theta^{n+1}-\eta^n) + \frac{C}{(n+1)}
%  \int_0^1 \dk ((1-s)\eta^n+s\eta^{n+1}, \eta^n)\ ds 
\end{split}
\end{equation}
Let us estimate the right-hand side of the inequality. For any $s\in [0,1]$,  Lemma \ref{lem:repDPhi} states that
\begin{equation}
\begin{split}
\frac{\delta \Phi}{\delta \eta} ((1-s)\eta^n+s\eta^{n+1})(\theta^{n+1}-\eta^n) &= \int_\Gamma J(\gamma(0), \gamma, (1-s)\eta^n+s\eta^{n+1})) \d(\theta^{n+1}-\eta^n)(\gamma)\\
%% \ds = \;  \int_\Gamma \int_0^T L(\gamma(t),\dot \gamma(t))\ dt d(\theta^{n+1}-\eta^n)(\gamma) + \int_0^T\int_{\states} f(\gamma(t), e_t\sharp ((1-s)\eta^{n+1}+s \eta^n)) \d t \d(\theta^{n+1}-\eta^n)(\gamma)\\
% \ds \qquad  +  \int_{\states} g(\gamma(T, e_T\sharp ((1-s)\eta^{n+1}+s \eta^n))\ d(\theta^{n+1}-\eta^n)(\gamma) \\
&= \int_\Gamma J(\gamma(0), \gamma, \eta^n) \d (\theta^{n+1}-\eta^n)(\gamma)+R(s)\\
% \leq \ds \; \int_\Gamma \int_0^T L(\gamma(t),\dot \gamma(t))\ dt d(\theta^{n+1}-\eta^n)(\gamma) + \int_0^T\int_{\states} f(\gamma(t), e_t\sharp \eta^n)\ dtd(\theta^{n+1}-\eta^n)(\gamma)\\
% \ds \qquad  +  \int_{\states} g(\gamma(T, e_T\sharp \eta^n)\ d(\theta^{n+1}-\eta^n)(\gamma) + R(s)
% C \int_0^T \dk(e_t\sharp ((1-s)\eta^{n+1}+s \eta^n)), e_t\sharp \eta^n) \ dt +  \dk(e_T\sharp ((1-s)\eta^{n+1}+s \eta^n)), e_T\sharp \eta^n) \ dt
\end{split}
\end{equation}
where, by the definition of $J$ and Lipschitz continuity of $f$ and $g$, 
\begin{equation}
\begin{split}\label{defR}
R(s) = & \int_\Gamma \int_0^T \big( f(\gamma(t), e_t\sharp((1-s)\eta^n+s\eta^{n+1})) - f(\gamma(t), e_t\sharp \eta^n) \big) \d t  \d (\theta^{n+1}-\eta^n)(\gamma) \\
&+ \int_\Gamma \big( g(\gamma(T), e_T\sharp((1-s)\eta^n+s\eta^{n+1})) - g(\gamma(T), e_T\sharp \eta^n)\big) \d (\theta^{n+1}-\eta^n)(\gamma) \\
\leq & \;  C \sup_{t\in [0,T]} \dk\left(e_t\sharp ((1-s)\eta^{n+1}+s \eta^n)), e_t\sharp \eta^n\right). 
\end{split}
\end{equation}
Note that, by the definition of $\dk$, we have for any $t\in [0,T]$,
$$
\begin{array}{l}
\ds \dk(e_t\sharp ((1-s)\eta^{n+1}+s \eta^n)), e_t\sharp \eta^n) \\
 \qquad \qquad \ds \leq  \sup_\xi \mint \xi(x)  \d (e_t\sharp ((1-s)\eta^{n+1}+s \eta^n)(x)- \mint \xi(x)\ d(e_t\sharp \eta^n)(x) \\
 \qquad \qquad \ds \leq (1-s) \sup_\xi \mint \xi(x)\  d(e_t\sharp \eta^{n+1})(x)- \mint \xi(x)\ d(e_t\sharp \eta^n)(x) \\
 \qquad \qquad \ds \leq \frac{(1-s)}{n+1} \sup_\xi \mint \xi(x)\  d (e_t\sharp (\theta^{n+1}-  \eta^n) )(x) \\
 \qquad \qquad \ds \leq \frac{(1-s)}{n+1} \sup_\xi \mint (\xi(x)-\xi(0))\  d(e_t\sharp \theta^{n+1}- e_t\sharp \eta^n)(x) \leq \frac{C}{n+1}, 
\end{array}
$$
where the supremum is taken over the set of Lipschitz maps $\xi: \states \to \R$ with Lipschitz constant not larger than $1$. Therefore 
$$
\Phi(\eta^{n+1})- \Phi(\eta^n) \leq \frac{1}{(n+1)} \int_\Gamma J(\gamma(0), \gamma, \eta^n)\ d(\theta^{n+1}-\eta^n)(\gamma)+\frac{C}{(n+1)^2},
$$
where the first term in the right-hand side is nonpositive thanks to Lemma \ref{lem.optibargamma}. 
\end{proof}

%%%%%%%%%%%%%%%%%%%%%%%%%%%%%%%
\subsection{Convergence of the Fictitious Play}

In this subsection, we prove Theorem \ref{th.mainOrdre1}. Recall that Lemma \ref{LipschitzBound} states that the sequence $(\eta^n)$ is tight. 
%
%We first show a property of cluster point of this sequence. 
%
%
%\begin{Lemma}
%For every arbitrary $ \eta_1 , \eta_2 \in \mes (\Gamma)$ we have:
%$$\sup_{t\in [0,T]} \dk (e_t \sharp \eta_1 , e_t \sharp \eta_2) \leq \dk (\eta_1 , \eta_2).$$
%\end{Lemma}
%\begin{proof}
%Fix $t\in [0,T]$. Consider an arbitrary 1-Lipschitz continuous function $\psi: \states \to \R$ and define $\tilde \psi : \Gamma \to \R$ as $\tilde \psi (\gamma) = \psi (\gamma(t))$. Obviously $\tilde \psi$ is 1-Lipschitz continuous since:
%$$ \abs \tilde \psi (\gamma_1) - \tilde \psi (\gamma_2) \abs = \abs \psi (\gamma_1(t)) - \psi (\gamma_2(t)) \abs \leq \abs \gamma_1(t) - \gamma_2(t) \abs \leq \norm \gamma_1 - \gamma_2 \norm_{\infty}.$$
%We then have:
%$$\leftnorm \mint \psi (x) \d (e_t \sharp \eta_2 - e_t \sharp \eta_1)(x) \rightnorm =\leftnorm \int_{\Gamma} \tilde \psi (\gamma) \d (\eta_2 - \eta_1)(\gamma) \rightnorm \leq \dk(\eta_1 , \eta_2).$$
%Since $\psi$ is arbitrary, this proves that $\dk (e_t\sharp \eta_1 , e_t\sharp \eta_2) \leq \dk (\eta_1 , \eta_2) $, which completes our proof by taking $\sup$ for all $t\in [0,T]$.
%\end{proof}
We next Lemma characterizes the cluster distribution :
\begin{Lemma}\label{lem:key} Any cluster point $\bar \eta$ of the sequence $(\eta^n)$ satisfies  
\begin{equation}\label{eq.dPhideta}
\frac{\delta \Phi}{\delta \eta}(\bar \eta)(\bar \eta) \leq \frac{\delta \Phi}{\delta \eta}(\bar \eta)(\theta) \qquad \forall \theta \in P_0(\Gamma),
\end{equation}
which means that  $\bar \eta-$a.e. $\gamma$ is optimal for the map $\tilde \gamma \to J(\gamma(0), \tilde \gamma, \bar \eta)$ under the constraint $\tilde \gamma(0)=\gamma(0)$. 
 \end{Lemma}

\begin{proof}
Let us define:
$$a^{n+1}:= -\frac{\delta \Phi}{\delta \eta} (\eta^n)(\theta^{n+1}-\eta^n) = -\int_{\Gamma} J(\gamma(0) , \gamma , \eta^n) \d (\theta^{n+1} - \eta^n)$$ 
$$= \int_{\Gamma} J(\gamma(0) , \gamma , \eta^{n}) \d \eta^{n} (\gamma) - \min_{\theta \in \mes_0 (\states)} \int_{\Gamma} J(\gamma(0) , \gamma , \eta^{n}) \d \theta(\gamma),$$
where the last equality come from Lemma \ref{lem.optibargamma}.
Then according to Proposition \ref{prop:Phi-Phi} the sequence $(a^n)$ is non-negative and, by \eqref{kjhbze}, the quantity $ \sum_k a^k/k$ is finite (because $\Phi$ is bounded below). Therefore by Lemma \ref{Convtozero} we have:
\begin{equation}\label{MondShap}
\lim_{N\to+\infty} \frac{1}{N} \sum_{k=1}^N a^k = 0.
\end{equation}
Let us now check that $a_n\leq C/n$ for some constant $C$. By arguments similar to the ones in the proof of Proposition \ref{prop:Phi-Phi}, we have, for any $\theta \in P_0(\Gamma)$,  
\be\label{kjhbze2}
\left|\frac{\delta \Phi}{\delta \eta} (\eta^n)(\theta)- \frac{\delta \Phi}{\delta \eta} (\eta^{n+1})(\theta)\right| \leq \frac{C}{n}.
\ee
On the other hand, by optimality of $\theta^{n+1}$ and $\theta^{n+2}$ in Lemma \ref{lem.optibargamma} and \eqref{kjhbze2}, we have
$$
\begin{array}{rl}
\ds \frac{\delta \Phi}{\delta \eta} (\eta^n)(\theta^{n+1}) \; = & \ds \min_{\theta \in \mes_0 (\states)} \int_{\Gamma} J(\gamma(0) , \gamma , \eta^{n}) \d \theta(\gamma)
%\\
\leq 
%& \ds 
\int_{\Gamma} J(\gamma(0) , \gamma , \eta^{n}) \d \theta^{n+2}(\gamma) \\
& \\
\leq & \ds \int_{\Gamma} J(\gamma(0) , \gamma , \eta^{n+1}) \d \theta^{n+2}(\gamma) +C/n = \frac{\delta \Phi}{\delta \eta} (\eta^{n+1})(\theta^{n+2}) +C/n\\
= & \ds \min_{\theta \in \mes_0 (\states)} \int_{\Gamma} J(\gamma(0) , \gamma , \eta^{n+1}) \d \theta(\gamma)+C/n
 \\
\leq & \ds \int_{\Gamma} J(\gamma(0) , \gamma , \eta^{n+1}) \d \theta^{n+1}(\gamma) +C/n
%\\
= 
%& \ds 
\frac{\delta \Phi}{\delta \eta} (\eta^n)(\theta^{n+1})+C/n,
\end{array}
$$
which proves that 
$$
\left| \frac{\delta \Phi}{\delta \eta} (\eta^n)(\theta^{n+1})- \frac{\delta \Phi}{\delta \eta} (\eta^{n+1})(\theta^{n+2})\right|
\leq C/n.
$$
So we have:
\begin{equation*}
\begin{split}
\left|a^n-a^{n+1}\right| &=  \left|\frac{\delta \Phi}{\delta \eta} (\eta^{n})(\eta^{n} - \theta^{n+1})  - \frac{\delta \Phi}{\delta \eta} (\eta^{n+1})(\eta^{n+1} - \theta^{n+2})\right| \\
&\leq  \left|\frac{\delta \Phi}{\delta \eta} (\eta^{n})(\eta^{n})- \frac{\delta \Phi}{\delta \eta} (\eta^{n+1})(\eta^{n+1})\right|
+  \left| \frac{\delta \Phi}{\delta \eta} (\eta^n)(\theta^{n+1})- \frac{\delta \Phi}{\delta \eta} (\eta^{n+1})(\theta^{n+2})\right| \\
& \leq \left|\frac{\delta \Phi}{\delta \eta} (\eta^{n})(\eta^{n}-\eta^{n+1})\right|+C/n
=  \frac{1}{n+1} \left| \frac{\delta \Phi}{\delta \eta} (\eta^n)(\theta^{n+1}-\eta^n)\right|+C/n \leq C/n.
\end{split}
\end{equation*}
By \eqref{MondShap} and the above estimate, we conclude that $a_n \to 0$ thanks to Lemma \ref{Convtozero}.

Let now $\bar \eta$ be any cluster point of the sequence $(\eta^n)$. Let us check that \eqref{eq.dPhideta} holds.  
Let $\theta \in \mes_0 (\states)$. Then, from Lemma \ref{lem.optibargamma}, for every $n \in \N$ we have:
$$
\frac{\delta \Phi}{\delta \eta}(\eta^n)(\eta^n) - a_n = \frac{\delta \Phi}{\delta \eta}(\eta^n)(\theta^{n+1}) \leq \frac{\delta \Phi}{\delta \eta}(\eta^n)(\theta).$$
If $( \eta^{n_i})_{i\in \N} $ is such that $\eta^{n_i} \to \bar \eta $, then:
$$\forall \gamma \in \Gamma : \quad \leftnorm J(\gamma , \gamma(0) , \bar \eta ) - J(\gamma , \gamma(0) , \eta^{n_i} )  \rightnorm \leq K \sup _{t\in [0,T]} \dk (e_t \sharp \eta^{n_i} , e_t \sharp \bar \eta) ,$$
where the last term tends to $0$ because the maps $t\to e_t \sharp \eta^{n_i}$ are uniformly continuous (from Lemma \ref{LipschitzBound}) and converges pointwisely (and thus uniformly) to   $t\to e_t \sharp \bar \eta$. This yields that $\ds (\frac{\delta \Phi}{\delta \eta}(\eta^{n_i})(\theta))$ converges to $\ds\frac{\delta \Phi}{\delta \eta}(\bar \eta)(\theta)$. On the other hand, by lower semicontinuity of the map $\gamma\to J(\gamma,\gamma(0),\bar \eta)$ on $\Gamma$,  we have  
$$
\frac{\delta \Phi}{\delta \eta}(\bar \eta)(\bar \eta)\leq \liminf \frac{\delta \Phi}{\delta \eta}(\bar \eta)(\eta^{n_i})= \liminf \frac{\delta \Phi}{\delta \eta}(\eta^{n_i})(\eta^{n_i}),
$$
which proves \eqref{eq.dPhideta}. 

Let us check that $\bar \eta-$a.e. $\gamma$ is optimal for the map $\tilde \gamma \to J(\gamma(0), \tilde \gamma, \bar \eta)$ under the constraint $\tilde \gamma(0)=\gamma(0)$. Let $\theta= \mint \delta_{\bar \gamma_x}m_0(x)\d x$ where $\bar \gamma_x$ is (a measurable selection of) an optimal solution for $\tilde \gamma\to J(x,\tilde \gamma, \bar \eta)$ under the constraint $\tilde \gamma(0)=x$. If we disintegrate $\bar \eta$ into $\bar \eta= \mint \bar \eta_x m_0(x)\d x$, then, for $m_0-$a.e. $x$ and $\bar \eta_x-$a.e. $\gamma$ we have 
\be\label{liuqnesdfclm}
J(x,\bar \gamma_x,\bar \eta)\leq J(x,\gamma, \bar \eta).
\ee
Integrating over $\bar \eta_x$ and then against $m_0$ then implies that 
$$
\frac{\delta \Phi}{\delta m}(\bar \eta)(\theta)= \mint J(x,\bar \gamma_x,\bar \eta)m_0(x) \d x \leq 
 \int_{\Gamma} J(\gamma(0),\gamma,\bar \eta) \d \bar \eta(\gamma)=  \frac{\delta \Phi}{\delta m}(\bar \eta)(\bar \eta).
 $$
 As the reverse inequality always holds, this proves that there must be an equality in \eqref{liuqnesdfclm} a.e., which proves the claim. 
\end{proof}

\begin{proof}[Proof of Theorem \ref{th.mainOrdre1}.] Let $(\bar \eta,\bar \theta)$ be the limit of a converging subsequence $(\eta^{n_i},\theta^{n_i})$. We set $$
\bar u (t,x):= \inf_{\gamma\in \Gamma, \ \gamma(t)=x} J(t,x,\gamma, \bar \eta) \qquad {\rm and}\qquad 
\bar m(t):= e_t \sharp \bar \eta. 
$$
By standard argument in optimal control, we know that $\bar u$ is a viscosity solution to \eqref{MFG}-(i) with terminal condition $\bar u(T,x)= g(x,\bar m(T))$. Moreover, $\bar u$ is Lipschitz continuous and semiconcave (cf. for instance Lemma 5.2 in \cite{Cdga13}). 

It remains to check that $\bar m$ satisfies \eqref{MFG}-(ii). By Lemma \ref{lem:key}, we know that
$$
\frac{\delta \Phi}{\delta \eta}(\bar \eta)(\bar \eta) \leq \frac{\delta \Phi}{\delta \eta}(\bar \eta)(\theta) \qquad \forall \theta \in P_0(\Gamma),
$$
which means that $\bar \eta-$a.e. $\gamma$ is optimal for the map $\tilde \gamma \to J(\gamma(0), \tilde \gamma, \bar \eta)$ under the constraint $\tilde \gamma(0)=\gamma(0)$. 
%
%
%By disintegration Theorem (see Theorem 5.3.1 in \cite{AGS}) we can disintegrate $\bar \eta$ into $\bar \eta= \mint \bar \eta_x dm_0(x)$ so that, for $m_0-$a.e. $x$, 
%$$\bar \eta_x \in \mes(\Gamma) , \quad \text{supp}(\bar \eta_x) \subset \{ \gamma \in \Gamma \abs \gamma (0)=x \}.$$
%Combining \eqref{eq.dPhideta} and Lemma \ref{lem:repDPhi}, we have therefore, for $m_0-$a.e. $x\in \states$, 
%$$
%J(x,\tilde \gamma, \bar \eta) = \inf_{\gamma\in H^1, \ \gamma(0)=0} J(x,\gamma, \bar \eta) =\bar u(0,x)\qquad \mbox{\rm for $\bar \eta_x-$a.e. $\tilde \gamma$.}
%$$
Following Theorem 6.4.9 in \cite{CannarsaSinestrari}, the optimal solution for $J(x,\cdot, \bar \eta)$ is unique at  any point of differentiability of $\bar u(0,\cdot)$ (let us call it $\bar \gamma_x$). Disintegrating $\bar \eta$ into $\bar \eta= \mint \bar \eta_x dm_0(x)$, we have therefore, since $m_0$ is absolutely continuous,  
$$
\bar \eta_x = \delta_{\bar \gamma_x} \qquad \mbox{\rm for $m_0-$a.e. $x\in \states$,}
$$
so that 
\begin{equation}\label{identifbareta}
\bar\eta = \int_{\states} \delta_{\bar \gamma_x}m_0(x)\d x\qquad {\rm and} \qquad  \bar m(t) = \bar \gamma_\cdot(t)\sharp m_0 \qquad \forall t\in [0,T].
\end{equation}
Let us also recall  that the derivative of $\bar u(t,\cdot)$ exists along the optimal solution $\bar \gamma_x$ and that
$$
\dot{\bar \gamma}_x(t) = -D_pH(\bar \gamma_x(t), \nabla \bar u(t, \bar \gamma_x(t))\qquad \forall t\in (0,T]
$$
(see  Theorems 6.4.7 and 6.4.8  of \cite{CannarsaSinestrari}). This proves that $\bar m$ is a solution in the sense of distribution  of \eqref{MFG}-(ii) (where we denote by $\nabla \bar u$ any fixed Borel measurable selection of the map $(t,x)\to D^*u(t,x)$, the set of reachable gradients of $u$ at $(t,x)$, see \cite{CannarsaSinestrari}). 
Proposition \ref{prop:mbounded} in appendix states  that \eqref{MFG}-(ii) has a unique solution and that this solution has a density in $L^\infty$: thus $\bar m$ is in $L^\infty$, which shows that the pair $(\bar u,\bar m)$ is a solution of the MFG system \eqref{MFG}.

%A first consequence of this is that the cluster point $\bar \eta$ is independent of the chosen subsequence. Indeed, since from its very definition the dependence with respect to $\bar \eta$ of $J(x,\gamma, \bar \eta)$ is only through the family of measures $(\bar m(t)=e_t\sharp \bar \eta)$ and since $\bar m$ is uniquely defined, $J(x,\gamma, \bar \eta)$ is independent of the choice of the subsequence: we denote it $J(x,\gamma, \bar m)$. Then $\bar \gamma_x$ defined above is also independent of the subsequence, which characterizes $\bar \eta$ in a unique way thanks to \eqref{identifbareta}. Therefore the entire sequence $(\eta^n)$ converges to $\bar \eta$. 

In order to identify the cluster point $\bar \theta$, let us recall that $\theta^n$ is defined by 
$$
\theta^n= \bar \gamma^n_\cdot \sharp m_0,
$$
where, for any $x\in \states$, $\bar \gamma^n_x$ is a minimum of $J(x,\cdot, \eta^n)$ under the constraint $\gamma(0)=x$. As the criterion $J(x,\cdot, \eta^{n_i})$ $\Gamma-$converges to $J(x,\cdot, \bar \eta)$ and since at any point of differentiability of $\bar u(0,\cdot)$ the optimal solution $\bar \gamma_x$ is unique, standard compactness arguments show that $(\bar \gamma^{n_i}_x)$ converges to $\bar \gamma_x$ for a.e. $x\in \T^d$. Therefore $(\theta^{n_i})$ converges  to $\bar \gamma_\cdot \sharp m_0$, which is nothing but $\bar \eta$ by \eqref{identifbareta}. So we conclude that $\bar \theta=\bar \eta$.

Finally, if \eqref{mono} holds, then we claim that $\bar \eta$ is independent of the chosen subsequence. Indeed, since from its very definition the dependence with respect to $\bar \eta$ of $J(x,\gamma, \bar \eta)$ is only through the family of measures $(\bar m(t)=e_t\sharp \bar \eta)$ and since, by \eqref{mono}, there exists a unique solution to the MFG system and thus $\bar m$ is uniquely defined, $J(x,\gamma, \bar \eta)$ is independent of the choice of the subsequence. Then $\bar \gamma_x$ defined above is also independent of the subsequence, which characterizes $\bar \eta$ in a unique way thanks to \eqref{identifbareta}. Therefore the entire sequence $(\eta^n,\theta^n)$ converges to $(\bar \eta,\bar \eta)$. 
\end{proof}

\begin{Remark}\label{rem.eqmeas} The proof shows that a measure $\bar \eta\in \mes_0(\Gamma)$ which satisfies \eqref{eq.dPhideta} can be understood as the representation of a MFG equilibrium. Indeed, if we define $(\bar u,\bar m)$ as in \eqref{e.defbarmbaru}, then $(\bar u,\bar m)$ is a solution to the MFG system \eqref{MFG}. Conversely, if $(\bar u,\bar m)$ is a solution to the MFG system \eqref{MFG}, then the relation \eqref{identifbareta} identifies uniquely a measure $\bar \eta\in \mes_0(\Gamma)$. For this reason, we call such a measure an equilibrium measure. 
\end{Remark}

%%%%%%%%%%%%%%%%%%%%%%%%%%
\subsection{The Learning Procedure in $N$-Players games}
In this part we show that the Fictitious Play in the Mean Field Game with large (but finite) number of players $N \in \N$ converges in some sense to the equilibrium of our Mean Field Game with infinite number of players. For every $N \in \N$, fix a sequence of initial states $ x_1 ^N , x_2 ^N , \cdots , x_N ^N \in \states$ such that:
$$\lim_{N \to \infty} \dk ( m_0 ^N , m_0) = 0$$
where $\ds m_0 ^N = \frac{1}{N} \sum_{i+1}^N  \delta_{x_i^N}$ is the empirical measure associated with the $\{x^N_i\}_{i=1, \dots, N}$. As in the case of an infinite population, let us define the sequences $\eta^{n,N}, \theta ^{n,N} \in \mes (\Gamma)$, for $n \in \N ^*$ in the following way:
\begin{equation}\label{LMFGN}
\begin{split}
\eta ^{n+1,N} = \frac{1}{n+1} (\theta ^{1,N} + \theta ^{2,N} + \cdots + \theta ^{n+1,N}) \\
\theta ^{n+1,N} = \frac{1}{N} (\delta_{\gamma_{x_1 ^N} ^{n+1 , N}} + \delta_{\gamma_{x_2 ^N} ^{n+1,N}} + \cdots + \delta_{\gamma_{x_N ^N} ^{n+1,N}}) \\
\end{split}
\end{equation}
where $\gamma_{x_i ^N} ^{n+1,N}$ is an optimal path which minimizes $J(x_i ^N , \cdot , \eta ^{n,N} )$. As before  one can show that if
$$a^{n+1,N}:= -\frac{\delta \Phi}{\delta \eta} (\eta^{n,N})(\theta^{n+1,N}-\eta^{n,N}) = -\int_{\Gamma} J(\gamma(0) , \gamma , \eta^{n,N}) \d (\theta^{n+1,N} - \eta^{n,N})(\gamma)$$
$$= \int_{\Gamma} J(\gamma(0) , \gamma , \eta^{n,N}) \d \eta^{n,N} (\gamma) - \min_{\theta \in \mes(\Gamma), e_0\sharp \theta = m_0 ^N} \int_{\Gamma} J(\gamma(0) , \gamma , \eta^{n,N}) \d \theta(\gamma) ,$$
then we have $\lim_{n \to \infty} a^{n,N} = 0$. This proves that any accumulation distribution $\bar{\eta}^N$  of the sequence $\{ \eta^{n,N} \} _{n\in \N ^*}$ satisfies:
\begin{equation}\label{AccN}
\int_{\Gamma} J(\gamma (0) , \gamma , \bar{\eta}^N) \d \bar{\eta}^N(\gamma) = \min_{\theta \in \mes(\Gamma), e_0 \sharp \theta = m_0 ^N}  \int_{\Gamma} J(\gamma (0) , \gamma , \bar{\eta}^N) \d \theta(\gamma).
\end{equation}
So if $\ds \bar{\eta}^N = \frac{1}{N}(\bar{\eta}_{x_1}^N + \bar{\eta}_{x_2}^N + \cdots + \bar{\eta}_{x_N}^N)$ then
$$\text{supp} (\bar{\eta}_{x_i}^N ) \subseteq \text{argmin}_{\gamma(0)=x_i} J(x_i , \gamma , \bar{\eta}^N) .$$
Note that, in contrast with the case of an infinite population, this is not an equilibrium condition, since the deviation of a player changes the measure $\bar \eta^N$ as well. 

In the following Theorem we prove that any accumulation point $\bar{\eta}$ of $\{ \bar{\eta}_N \}$ satisfies:
\begin{equation}\label{e.kjhqsdbd}
\int_{\Gamma} J(\gamma(0) , \gamma , \bar{\eta}) \d \bar{\eta} (\gamma) = \min_{\theta \in \mes_0(\Gamma)} \int_{\Gamma} J(\gamma(0) , \gamma , \bar{\eta}) \d \theta(\gamma) ,
\end{equation}
where $\mes_0(\Gamma)$ is the set of measure $\theta\in \mes(\Gamma)$ such that $e_0\sharp \theta=m_0$. We have seen in Remark \ref{rem.eqmeas} that this condition characterizes an MFG equilibrium.

\begin{Theorem} Assume that   \eqref{hypm0}, \eqref{hyp:unifCv}, \eqref{hyp:Hcondsupp}, \eqref{regucondf} and \eqref{potpot2} hold.
Consider the Fictitious Play for the $N-$player game as described in \eqref{LMFGN} and let $\bar{\eta}^N$ by an accumulation distribution of $(\eta^{n,N})_{n\in \N}$. Then every accumulation point of pre-compact set of $\{ \bar{\eta}^N \}_{N \in \N}$ is an MFG equilibrium. 

If furthermore  the monotonicity condition \eqref{mono} holds, then $(\bar \eta^N)$ has a limit which is the MFG equilibrium. 
\end{Theorem}
\begin{proof} Consider $\bar{\eta}$ as an accumulation point of the set $\{ \bar{\eta}^N \}_{N \in \N}$. It is sufficient to show that for every $\theta \in \mes (\Gamma)$ such that $e_0 \sharp \theta = m_0$, we have
\begin{equation}\label{Acc}
\int_{\Gamma} J(\gamma(0) , \gamma , \bar{\eta}) \d \bar{\eta}(\gamma) \leq \int_{\Gamma} J(\gamma(0) , \gamma , \bar{\eta}) \d \theta (\gamma).
\end{equation}
Since $m_0$ is absolutely continuous with respect to the Lebesgue measure, there exists an optimal transport map $\tau_N : \states \to \states$ such that:
$$ \tau_N \sharp m_0 = m_0 ^N ,\quad \dk (m_0 , m_0 ^N) = \mint |x- \tau_N(x)| \d m_0 (x)$$ 
(see \cite{AGS}). We define the functions $\xi_N : \Gamma \to \Gamma$ as follows:
$$\xi_N(\gamma) = \gamma -\gamma(0) + \tau_N(\gamma(0))$$
and set $\theta ^N = \xi_N \sharp \theta$. Then we have $$e_0 \sharp \theta ^N = e_0 \sharp (\xi_N \sharp \theta) = (e_0  \circ \xi_N) \sharp \theta = (\tau_N \circ e_0) \sharp \theta = \tau_N \sharp ( e_0 \sharp \theta) = \tau_N \sharp m_0 = m_0 ^N  .$$
Then the characterization \eqref{AccN} of $\bar{\eta}^N$ yields:
\be\label{Acc2}
\int_{\Gamma} J(\gamma(0) , \gamma , \bar{\eta}^N) \d \bar{\eta}^N(\gamma) \leq \int_{\Gamma} J(\gamma(0) , \gamma , \bar{\eta}^N) \d \theta^N (\gamma).\ee
By lower semicontinuity of $J$ we have 
$$
\int_{\Gamma} J(\gamma(0) , \gamma , \bar{\eta}) \d \bar{\eta}(\gamma)\leq \liminf_N
\int_{\Gamma} J(\gamma(0) , \gamma , \bar{\eta}^N) \d \bar{\eta}^N(\gamma).
$$
On the other hand, by the definition of $\xi^N$ and $\theta^N$ and  the decomposition $\theta=\int_{\T^d} \theta_xm_0(x)\d x$, we have 
$$
\begin{array}{ll}
\ds \int_{\Gamma} J(\gamma(0) , \gamma , \bar{\eta}^N) \d \theta^N (\gamma)\\ 
\qquad \ds = 
 \int_{\T^d}  \int_\Gamma (\int_0^T L(\gamma(t) -\gamma(0) + \tau_N(\gamma(0)) , \dot \gamma(t)) + f(\gamma(t) -\gamma(0) + \tau_N(\gamma(0)), e_t\sharp \bar \eta^N) \ \d t \\  
\qquad \ds \qquad \qquad \qquad \qquad \ds + g(\gamma(t) -\gamma(0) + \tau_N(\gamma(0)), e_T\sharp \bar \eta^N)) m_0(x) \d \theta_x (\gamma) \d x,
\end{array} $$
where, by dominate convergence, the right-hand side converges to the right-hand side of  \eqref{Acc}. 
So letting $N\to \infty$ in \eqref{Acc2} gives exactly \eqref{Acc}. 

Under \eqref{mono}, the MFG equilibrium is unique. Hence, for any $\epsilon > 0$ there exists $N_\ep \in \N$ such that for any $N > N_\ep$ and any accumulation point $\bar{\eta}^N$ we have $\dk (\bar{\eta} , \bar{\eta}^N ) <\epsilon$.
\end{proof}
\begin{Corollary} Assume \eqref{hypm0}, \eqref{hyp:unifCv}, \eqref{hyp:Hcondsupp}, \eqref{regucondf} and \eqref{potpot2} and \eqref{mono}.
Then, for any $\epsilon >0$ there is $N_\ep \in \N$ such that for any $N > N_\ep$,
$$ \exists n(N,\epsilon) \in \N: \quad \forall n > n(N,\epsilon): \quad \dk (\eta ^{n,N} , \bar \eta) < \epsilon, $$
where $\bar \eta$ is the MFG equilibrium. In other words, for every $\epsilon > 0$, one can reach to the $\epsilon-$neighborhood of the equilibrium point if the number of players $N$ is large enough. 
\end{Corollary}

%%%%%%%%%%%%%%%%%%%%%%%%%%%%
\appendix
\section{Well-posedness of a continuity equation} 

We consider the continuity equation
\begin{equation}\label{e.FP1erOrdre}
\left\{\begin{array}{l}
\partial_t m -{\rm div}(m D_pH(x,\nabla \bar u))=0\qquad {\rm in}\; (0,T)\times \T^d\\
m(0,x)=m_0(x). 
\end{array}\right.
\end{equation}
where $\bar u$  is the viscosity solution to 
$$
\left\{\begin{array}{l}
\ds -\partial_t u  +H(x,\nabla u(t,x)) =f(x,\bar m(t)), \quad (t,x)\in  [0,T]\times \states  \\
u(T,x)=g(x, m(T)), \quad  x\in \states
\end{array}\right.
$$
Let us recall that $\bar u$ is semi-concave. In \eqref{e.FP1erOrdre} we denote by $\nabla \bar u$ any fixed Borel measurable selection of the map $(t,x)\to D^*u(t,x)$ (the set of reachable gradients of $u$ at $(t,x)$, see \cite{CannarsaSinestrari}). 
The section is devoted to the proof of the following statement.

\begin{Proposition} \label{prop:mbounded} There exists a unique solution $\bar m$ of \eqref{e.FP1erOrdre} in the sense of distribution. Moreover $\bar m$ is absolutely continuous and satisfies
$$
\sup_{t\in [0,T]} \|\bar m(t,\cdot)\|_\infty \leq C.
$$
\end{Proposition}

The difficulty for the proof comes from the fact that the vector field $-D_pH(t,x,\nabla u)$ is not smooth: it is even discontinuous in general. The analysis of transport equations with non smooth vector fields  has attracted a lot of attention since the DiPerna-Lions seminal paper \cite{DL89}. We face here a simple situation where the vector field generates almost everywhere a unique solution. Nevertheless uniqueness of solution of the associated continuity equation requires the combination of several arguments. We rely here on Ambrosio's approach \cite{Ambrosio, Ambrosio2}, in particular for the ``superposition principle'' (see Theorem \ref{theo:superposition} below).  

Let us start with the existence of a bounded solution to \eqref{e.FP1erOrdre}: this is the easy part. 

\begin{Lemma} There exists a solution to \eqref{e.FP1erOrdre} which belongs to $L^\infty$. 
\end{Lemma}

\begin{proof} We follow (at least partially) the perturbation argument given in the proof of Theorem 5.1 of \cite{Cdga13}. For $\ep>0$, let $(u^\ep,m^\ep)$ be the unique classical solution to 
$$
\left\{\begin{array}{l}
-\partial_t u^\ep -\ep\Delta u^\ep +H(x,\nabla u^\ep)= f(x,\bar m(t))\qquad {\rm in}\; (0,T)\times \T^d\\
\partial_t m^\ep-\ep\Delta m^\ep -{\rm div}(m^\ep D_pH(x,\nabla u^\ep))=0\qquad {\rm in}\; (0,T)\times \T^d\\
m^\ep(0,x)=m_0(x), \; u^\ep(T,x)= g(x,\bar m(t))\qquad {\rm in}\;  \T^d
\end{array}\right.
$$
Following the same argument as in \cite{Cdga13}, we know that the $(m^\ep)$ are uniformly bounded in $L^\infty$: there exists $C>0$ such that 
$$
\|m^\ep\|_\infty\leq C \qquad \forall \ep>0.
$$
Moreover (by semi-concavity) the $(\nabla u^\ep)$ are uniformly bounded and converge a.e. to $\nabla \bar u$ as $\ep$ tends to $0$. 
Letting $\ep\to 0$, we can extract a subsequence such that $m^\ep$ converges in $L^\infty-$weak* to a solution $m$ of \eqref{e.FP1erOrdre}.  
\end{proof}

The difficult part of the proof of Proposition \ref{prop:mbounded} is to check that the solution to \eqref{e.FP1erOrdre} is unique.
Let us first point out some basic properties of the solution $\bar u$: we already explained that $\bar u$ is Lipschitz continuous and semiconcave in space for any $t$, with a modulus bounded independently of $t$. We will repetitively use the fact that $\bar u$  can be represented  as the value function of a problem of calculus of variation: 
\be\label{RepForm}
\bar u(t,x)= \inf_{\gamma, \ \gamma(t)=x} \int_t^T \tilde L(s,\gamma(s),\dot \gamma(s),\bar m(s))ds + \tilde g(\gamma(T))
\ee
where we have set, for simplicity of notation,
$$
\tilde L(s,x,v) = L(x,v)+ f(x,\bar m(s)), \qquad  \tilde g(x)= g(x, \bar m(T)).
$$
For $(t,x)\in [0,T)\times \T^d$ we denote by ${\mathcal A}(t,x)$ the set of optimal trajectories for the control problem \eqref{RepForm}. 
%One easily checks that such set is nonempty, and that, if $(t_n,x_n)\to (t,x)$ and 
%$\gamma_n\in {\mathcal A}(t_n,x_n)$, then, up to some subsequence, $(\gamma_n)$ weakly converges in $H^1$ to
%some $\gamma\in {\mathcal A}(t,x)$. 

%Let us recall that, if $\gamma$ belongs to ${\mathcal A}(t,x)$, then $\gamma$ satisfies the Euler-Lagrange optimality condition: $\gamma$ is of class ${\mathcal C}^2$ on $[t,T]$ with
%\be\label{EulerCond}
%\frac{d}{dt} \; D_p\tilde L(s,\gamma(s),\dot\gamma(s)) = 
%D_x\tilde L(s,\gamma(s),\dot\gamma(s)) \qquad \forall s\in [t,T], 
%\ee
%and terminal condition
%\be\label{EulerCondTC}
%\qquad \dot \gamma(T)=-D_p\tilde H(T,\gamma(T), D\tilde g(\gamma(T))). 
%\ee

We need to analyze precisely the connexion between the differentiability of $\bar u$ with respect to the $x$ variable and
the uniqueness of the minimizer in  (\ref{RepForm}) (see \cite{CannarsaSinestrari}, Theorems 6.4.7 and 6.4.9 and Corollary 6.4.10).  Let $(t,x)\in [0,T]\times \T^d$ and $\gamma\in \Gamma$.  Then 
\begin{enumerate}
\item (Uniqueness of the optimal control along optimal trajectories) Assume that $\gamma \in {\mathcal A}(t,x)$. Then, for any $s\in (t,T]$, $\bar u(s,\cdot)$ is differentiable at $\gamma(s)$ for $s\in (t,T)$ and one has $\dot \gamma(s)= -D_p H(\gamma(s),\nabla u(s,\gamma(s)))$.

\item (Uniqueness of the optimal trajectories) 
$\nabla u(t,x)$ exists if and only if ${\mathcal A}(t,x)$ is a reduced to singleton. In this case, $\dot \gamma(t)= -D_p H(x,\nabla \bar u(t,x))$
where ${\mathcal A}(t,x)=\{\gamma\}$. 

\item (Optimal synthesis) conversely, if $\gamma(\cdot)$ is an absolutely continuous
solution of the differential equation
\be\label{eqdiff-Du}
\left\{\begin{array}{l}
\dot \gamma(s)=-D_p H(s,\gamma(s),\nabla \bar u(s,\gamma(s))) \qquad \mbox{\rm a.e. in }[t,T]\\
\gamma(t)=x,
\end{array}\right.
\ee
then the trajectory $\gamma$  is optimal for $\bar u(t,x)$. In particular, if $\bar u(t, \cdot)$ is differentiable at $x$, then equation (\ref{eqdiff-Du}) has a unique solution, corresponding to the optimal trajectory.
\end{enumerate}

The next ingredient is Ambrosio's superposition principle, which says that any weak solution to the transport equation
\be\label{eq:transport}
\partial_t \mu -{\rm div}(\mu D_p H(x,\nabla \bar u))=0\qquad {\rm in}\; (0,T)\times \T^d
\ee
can be represented by a measure on the space of trajectories of the ODE 
\be\label{ODE}
\dot \gamma (s)= -D_p H( \gamma(s), \nabla \bar u(s,\gamma(s)).
\ee
%For this let us define for any $t\in [0,T]$ the map $e_t:C^0([0,T],\T^d)\to \T^d$ by $e_t(\gamma)= \gamma(t)$ for $\gamma \in C^0([0,T],\T^d)$. 

\begin{Theorem}[Ambrosio superposition principle]\label{theo:superposition}  Let $\mu$ be a solution to \eqref{eq:transport}. Then there exists a Borel probability measure $\eta$ on $C^0([0,T],\T^d)$ such that $\mu(t)= e_t\sharp \eta$ for any $t$ and, for $\eta-$a.e. $\gamma\in C^0([0,T], \T^d)$, $\gamma$ is a solution to the ODE \eqref{ODE}.  
\end{Theorem}

See, for instance, Theorem 8.2.1. from \cite{AGS}. \\

%We will also need the notion of disintegration of a measure. 
%
%\begin{Theorem} \label{theo:desintegration}
%Let $X$ and $Y$ be two Polish spaces and $\lambda$ be a Borel probability measure on $X\times Y$. 
%Let us set $\mu=\pi_X\sharp \lambda$, where $\pi_X$ is the standard projection from $X\times Y$ onto $X$. 
%Then there exists a $\mu$-almost everywhere uniquely determined family of Borel probability measures $(\lambda_x)$ on $Y$ such that
%\begin{enumerate}
%\item the function $x \mapsto \lambda_{x}$ is Borel measurable, in the sense that $x \mapsto \lambda_{x} (B)$
% is a Borel-measurable function for each Borel-measurable set $B \subset Y$, 
% \item for every Borel-measurable function $f : X\times Y \to [0, +\infty]$,
%$$
%        \int_{X\times Y} f(x,y) d\lambda(x,y)\,  =\, \int_{X} \int_{Y} f(x,y) \, \mathrm{d} \lambda_{x} (y) \mathrm{d} \nu (x).
%$$
%\end{enumerate}
%\end{Theorem}
%
%See for instance the monograph \cite{AGS}. 
%
%

We are now ready to prove the uniqueness part of the result: 

\begin{proof}[Proof of Proposition \ref{prop:mbounded}.] Let $\mu$ be a solution of the transport equation \eqref{eq:transport}. From Ambrosio superposition principle, there exists a Borel probability measure $\eta$ on $C^0([0,T],\T^d)$ such that $\mu(t)= e_t\sharp \eta$ for any $t$ and, for $\eta-$a.e. $\gamma\in C^0([0,T], \T^d)$, $\gamma$ is a solution to the ODE $\dot \gamma = - D_pH(t,\gamma(t), \nabla u(t,\gamma(t)))$. As $m_0= e_0\sharp \eta$, we can disintegrate the measure $\eta$ into $\eta= \int_{\states} \eta_xdm_0(x)$, where $\gamma(0)=x$ for $\eta_x-$a.e. $\gamma$ and $m_0-$a.e. $x\in \T^d$. 
Since  $m_0$ is absolutely continuous,  for $m_0-$a.e. $x\in \T^d$, $\eta_x-$a.e. map $\gamma$ is a solution to the ODE starting from $x$. By the optimal synthesis explained above, such a solution $\gamma$ is optimal for the calculus of variation problem \eqref{RepForm}. As, moreover, for a.e. $x\in \T^d$ the solution of this problem is reduced to a singleton $\{\bar \gamma_x\}$, we can conclude that $\d \eta_x(\gamma)=\delta_{\bar \gamma_x}$ for $m_0-$a.e. $x\in \T^d$. Hence, for any continuous map $\phi:\T^d\to \R$, one has  
$$
\int_{\states}\phi(x) m(t,x))\d x= \int_{\states} \phi(\bar \gamma_x(t))m_0(x)\d x
$$
which defines $\mu$ in a unique way. 
\end{proof}

\end{document}